\documentclass[reqno]{amsart}
\usepackage[utf8x]{inputenc}  

\usepackage{setspace}
\usepackage[left=3.1cm, right=3.1cm, bottom=4cm]{geometry}                 

\usepackage{graphicx} 
\usepackage{pgf,tikz}
\usetikzlibrary{arrows}
\usepackage{amssymb}
\usepackage{pdfsync}
\usepackage{mathrsfs}
\usepackage{hyperref} 
\usepackage{verbatim} 
\usepackage{epstopdf}
\usepackage{bbm} 
\usepackage[colorinlistoftodos,prependcaption,textsize=tiny]{todonotes}
\usepackage{xargs}

\def\R{\mathbb{R}}

\def\C{\mathbb{C}}

\renewcommand{\d}{\text{\rm d}}

\def\wM {{\widetilde{M}}}

\newcommand{\mc}{\mathcal}

\newtheorem{theorem}{Theorem}

\newtheorem*{definition*}{Definition}

\newtheorem{proposition}[theorem]{Proposition}
\newtheorem{lemma}[theorem]{Lemma}

\makeatletter
\DeclareFontFamily{U}{tipa}{}
\DeclareFontShape{U}{tipa}{m}{n}{<->tipa10}{}
\newcommand{\arc@char}{{\usefont{U}{tipa}{m}{n}\symbol{62}}}%

\makeatother

\numberwithin{equation}{section}

\allowdisplaybreaks

\newcommand{\intav}[1]{\mathchoice {\mathop{\vrule width 6pt height 3 pt depth  -2.5pt
\kern -8pt \intop}\nolimits_{\kern -6pt#1}} {\mathop{\vrule width
5pt height 3  pt depth -2.6pt \kern -6pt \intop}\nolimits_{#1}}
{\mathop{\vrule width 5pt height 3 pt depth -2.6pt \kern -6pt
\intop}\nolimits_{#1}} {\mathop{\vrule width 5pt height 3 pt depth
-2.6pt \kern -6pt \intop}\nolimits_{#1}}}

\newcommand{\intavl}[1]{\mathchoice {\mathop{\vrule width 6pt height 3 pt depth  -2.5pt
\kern -8pt \intop}\limits_{\kern -6pt#1}} {\mathop{\vrule width 5pt
height 3  pt depth -2.6pt \kern -6pt \intop}\nolimits_{#1}}
{\mathop{\vrule width 5pt height 3 pt depth -2.6pt \kern -6pt
\intop}\nolimits_{#1}} {\mathop{\vrule width 5pt height 3 pt depth
-2.6pt \kern -6pt \intop}\nolimits_{#1}}}

\title[Sunrise strategy]{Sunrise strategy for 
the \\ continuity of maximal operators}

\author[Carneiro]{Emanuel Carneiro}
\author[Gonz\'{a}lez-Riquelme]{Cristian Gonz\'{a}lez-Riquelme}
\author[Madrid]{Jos\'{e} Madrid}

\address{
ICTP - The Abdus Salam International Centre for Theoretical Physics, 
Strada Costiera, 11, I - 34151, Trieste, Italy.}
\address{IMPA - Instituto de Matem\'{a}tica Pura e Aplicada, 
Estrada Dona Castorina, 110, Jardim Bot\^{a}nico, Rio de Janeiro - RJ, 22460-320, Brazil.}

\email{carneiro@ictp.it}
\email{carneiro@impa.br}

\address{IMPA - Instituto de Matem\'{a}tica Pura e Aplicada, Estrada Dona Castorina, 110, Jardim Bot\^{a}nico, Rio de Janeiro - RJ, 22460-320, Brazil.}
\email{cristian@impa.br}

\address{Department of  Mathematics,  University  of  California,  Los  Angeles (UCLA),  Portola Plaza 520, Los  Angeles,
California, 90095, USA.}
\email{jmadrid@math.ucla.edu}      

\date{\today}                                           
\begin{document}

\begin{abstract} In this paper we address the $W^{1,1}$--\,continuity of several maximal operators at the gradient level. A key idea in our global strategy is the decomposition of a maximal operator, with the absence of strict local maxima in the disconnecting set, into ``lateral" maximal operators with good monotonicity and convergence properties. This construction is inspired in the classical sunrise lemma in harmonic analysis. A model case for our sunrise strategy considers the uncentered Hardy-Littlewood maximal operator $\wM$ acting on $W^{1,1}_{\rm rad}(\R^d)$, the subspace of $W^{1,1}(\R^d)$ consisting of radial functions. In dimension $d\geq 2$ it was recently established by H. Luiro that the map $f \mapsto \nabla \wM f$ is bounded from $W^{1,1}_{\rm rad}(\R^d)$ to $L^1(\R^d)$, and we show that such map is also continuous. Further applications of the sunrise strategy in connection with the $W^{1,1}$--\,continuity problem include non-tangential maximal operators on $\R^d$ acting on radial functions when $d\geq 2$ and general functions when $d=1$, and the uncentered Hardy-Littlewood maximal operator on the sphere $\mathbb{S}^d$ acting on polar functions when $d\geq 2$ and general functions when $d=1$.
\end{abstract}

\maketitle 

\section{Introduction} 
\subsection{Background} For $f \in L^1_{\rm loc}(\R^d)$ we define its centered Hardy-Littlewood maximal function 
\begin{equation}\label{20200616:17:38}
Mf(x) := \sup_{r >0} \frac{1}{|B_r(x)|}\int_{B_r(x)} |f(y)|\,\d y = \sup_{r >0}\  \intav{B_r(x)} |f(y)|\,\d y\,,
\end{equation}
where $B_r(x) \subset \R^d$ is the open ball centered at $x$ with radius $r$, and $|B_r(x)|$ denotes its $d$-dimensional Lebesgue measure. The crossed integral symbol, as it appears on the right-hand side of \eqref{20200616:17:38}, will always mean an average over the domain of integration in this paper. The uncentered Hardy-Littlewood maximal function $\wM f$ is defined analogously to \eqref{20200616:17:38}, now taking the supremum over open balls that simply contain the point $x$ but that are not necessarily centered at $x$. Maximal operators like \eqref{20200616:17:38} are fundamental objects in harmonic analysis and partial differential equations, being useful tools in establishing a variety of pointwise convergence results.

\smallskip

The classical theorem of Hardy, Littlewood and Wiener states that $M: L^1(\R^d) \to L^{1,\infty}(\R^d)$ and $M: L^p(\R^d) \to L^p(\R^d)$, for $1 < p \leq \infty$, are bounded operators. Being a sublinear operator, the boundedness in $L^p(\R^d)$ ($1 < p \leq \infty$) plainly implies that $M: L^p(\R^d) \to L^p(\R^d)$ is also a continuous operator. The beautiful work of J. Kinnunen \cite{Ki} in 1997, a landmark in the regularity theory of maximal operators, establishes that $M: W^{1,p}(\R^d) \to W^{1,p}(\R^d)$ is bounded for $1 < p \leq \infty$; here $W^{1,p}(\R^d)$ is the first order Sobolev space with exponent $p$. The continuity of the map $M: W^{1,p}(\R^d) \to W^{1,p}(\R^d)$ ($1 < p \leq \infty$) is a non-trivial issue, since sublinearity is not in principle available at the gradient level. This question was settled, in the affirmative, only a decade later, in the elegant work of Luiro \cite{Lu1}. All the statements above hold for the uncentered version $\wM$ as well.

\smallskip

One of the striking features of the regularity theory of maximal operators is the appearance of positive boundedness results at the gradient level despite the absence of corresponding results at the function level. This is the sort of situation that may occur at the endpoint $p=1$. It is believed, for instance, that the total variation of $M f$ should control the total variation of $f$. This was formally posed in the work of Haj\l asz and Onninen \cite{HO} in 2004, in the following form: if $f \in W^{1,1}(\R^d)$, do we have that $Mf$ is weakly differentiable and 
$$\|\nabla M f\|_{L^1(\R^d)} \lesssim_d \, \|\nabla f\|_{L^1(\R^d)} \ ?$$
One can formulate the same question for $\wM$. This question remains unsolved in the general case, but there has been interesting partial progress, all in the affirmative. In dimension $d=1$ the question in the uncentered case was settled by Tanaka \cite{Ta} and Aldaz and P\'{e}rez L\'{a}zaro \cite{AP}, while the very subtle centered case was later settled by Kurka \cite{Ku}. In higher dimensions, Luiro \cite{Lu2} solved the problem in the uncentered case for functions $f \in W^{1,1}_{\rm rad}(\R^d)$, i.e. the subspace of $W^{1,1}(\R^d)$ consisting of radial functions. There are also a couple of promising new results by J. Weigt, solving the total variation version of this question for characteristic functions of sets of finite perimeter \cite{We1}, and its analogue for the dyadic maximal operator \cite{We2}. Related works in this topic include \cite{CFS, CGR, CS, HM, KL, LXK, PPSS, Ra, Saa}.

\smallskip

Once the boundedness is established, a natural question that arises is if the map $f \mapsto \nabla M f$ (or $f \mapsto \nabla \wM f$) from $W^{1,1}(\R^d)$ to $L^1(\R^d)$ is also continuous. Note the additional layer of difficulty coming from the fact that $f \mapsto M f$ (or $f \mapsto \wM f$) is not bounded in $L^1(\R^d)$. This endpoint continuity question was only settled, in the affirmative, in the uncentered case in dimension $d=1$ by Carneiro, Madrid and Pierce \cite[Theorem 1]{CMP}, bringing new oscillation-control mechanisms to overcome the additional obstacles inherent to the problem.

\subsection{Sunrise strategy: a model case} In this paper we aim to provide the next instalment in this theory. Our purpose here is to develop a strategy to approach the $W^{1,1}$--\,continuity problem for a certain class of maximal operators of general interest. Our first result, a model case for our global strategy, complements the recent boundedness result of Luiro \cite{Lu2}.

\begin{theorem}\label{Thm1}
The map $f \mapsto \nabla \wM f$ is continuous from $W^{1,1}_{\rm rad}(\R^d)$ to $L^1(\R^d)$ for $d \geq 2$.
\end{theorem}

Despite the innocence of the statement in Theorem \ref{Thm1}, one should not underestimate the subtlety of the problem, as it will become evident as the proof unfolds and we find ourselves in a beautiful maze of possibilities. It is worth mentioning a few words on the difficulties that one faces when trying to prove this theorem, in direct comparison to the core papers in the literature that deal with similar continuity issues. First, the original proof of Luiro \cite{Lu1} to show the continuity of $M$ (or $\wM$) in $W^{1,p}(\R^d)$ ($1 < p \leq \infty$) relies decisively on the boundedness of $M$ in $L^p(\R^d)$, which is not available in our situation. This was already an issue in the work of Carneiro, Madrid and Pierce \cite[Theorem 1]{CMP} to prove the continuity of $f \mapsto  \big(\wM f\big)'$ from $W^{1,1}(\R)$ to $L^1(\R)$, and a new path was developed. A crucial element in the proof of \cite[Theorem 1]{CMP} was the ability to decompose $\wM$ as a maximum of two operators, namely, 
\begin{equation}\label{20200616_18:28}
\wM f(x) = \max\big\{ M_Rf(x), M_Lf(x)\big\} \ \ {\rm for \ all } \ \ x \in \R,
\end{equation}
where $M_R$ and $M_L$ are the one-sided maximal operators, to the right and left, respectively. The monotonicity properties of these one-sided operators in the connecting and disconnecting sets played a very important role \cite[\S 5.4.1]{CMP}. In our situation of Theorem \ref{Thm1}, when dealing with radial functions on $\R^d$, there is no obvious way to decompose $\wM$ into two ``lateral" operators with similar monotonicity properties, and this is a major obstacle.

\smallskip

There is a parallel wave of very interesting results for the fractional Hardy-Littlewood maximal operator. For $0 < \beta < d$ we define the centered version as
\begin{equation}\label{20200616_12:55}
M_{\beta} f(x) = \sup_{r >0} \frac{1}{|B_r(x)|^{1 - \frac{\beta}{d}}}\int_{B_r(x)} |f(y)|\,\d y,
\end{equation}
whereas the uncentered version $\wM_{\beta}$ is defined analogously, with balls containing $x$ but not necessarily centered at $x$. In sympathy with the classical $L^p(\R^d) \to L^q(\R^d)$ bounds, it was proved by Kinnunen and Saksman \cite{KiSa} that $\wM_{\beta}, M_{\beta}:W^{1,p}(\R^d) \to W^{1,q}(\R^d)$ are bounded if $1 < p < \infty$, $0 < \beta < d/p$ and $q = dp / (d - \beta p)$. The continuity at this level was considered by Luiro in \cite{Lu3}. One has then the corresponding endpoint question (see \cite[Question B]{CMa}): is the operator $f \mapsto \nabla \wM_{\beta} f$ (or $f \mapsto \nabla M_{\beta} f$) bounded from $W^{1,1}(\R^d)$ to $L^{d/(d-\beta)}(\R^d)$? When $1 \leq \beta <d $ this question has a positive answer in general, as remarked in \cite{CMa} and the hard case is when $0 < \beta < 1$. The latter was answered in the affirmative in dimension $d=1$, for $\wM_{\beta}$ in \cite[Theorem 1]{CMa} and $M_{\beta}$ in \cite[Theorem 1.1]{BM1}, and in dimension $d >1$ for $f \in W^{1,1}_{\rm rad}(\R^d)$, also in both cases: for $\wM_{\beta}$ in \cite{LM} and for $M_{\beta}$ in \cite[Theorem 1.2]{BM1}. In this fractional endpoint case, whenever the boundedness holds, the $W^{1,1}$--\,continuity also holds. This was proved by Madrid in \cite{M2} (for $d=1$ and $\wM_{\beta}$) and Beltran and Madrid in \cite{BM, BM1} (in the other cases). Note here the presence of a certain smoothing effect, in the sense that the fractional maximal function \eqref{20200616_12:55} disregards balls of very small radii, and this plays a relevant role in such continuity proofs. The arguments in \cite{BM, BM1, M2} do not fully survive a passage to the limit $\beta \to 0^+$, and hence are also not amenable to treat the case of our Theorem \ref{Thm1}. 

\smallskip 

On the other hand, in \cite[Theorems 3 and 4]{CMP} one has some {\it negative} continuity results for $M_{\beta}$ and $\wM_{\beta}$ in ${\rm BV}(\R)$, despite having the corresponding positive boundedness result for $\wM_{\beta}$ \cite[Theorem 1]{CMa} (whereas the corresponding boundedness from ${\rm BV}(\R)$ to ${\rm BV}_q(\R)$ in the centered case is still an open problem; see \cite{CMa} for the precise formulation). A more classical example of an operator that is bounded in $W^{1,p}(\R^d)$ for $1 \leq p < \infty$, but is not continuous when $d\geq2$, is the symmetric decreasing rearrangement, as observed in the celebrated work of Almgren and Lieb \cite{AL}. This suggests that one should not, in principle, bet all chips on the validity of a continuity statement as in Theorem \ref{Thm1}.

\smallskip

Our approach will naturally draw some inspiration from these core continuity works \cite{BM, BM1, CMP, Lu1, Lu3, M2}, being perhaps a little more in line with the strategy of the first and third authors with L. Pierce in \cite{CMP}. In fact, the method developed here is more general and can be used to give an alternative proof of  \cite[Theorem 1]{CMP}, which is the one-dimensional case. The proof of Theorem \ref{Thm1} is carefully developed in Sections \ref{Sec_Preli} to \ref{Sect_Proof}, where each section addresses an independent aspect of the overall strategy. In Section \ref{Sec_Preli} we provide the preliminaries about maximal operators and radial Sobolev functions, and treat some basic regularity and convergence issues in this setup. In Section \ref{Sec_control} we establish a control of the convergence in a neighborhood of the origin, where potential singularities may appear, thus making it possible to concentrate our efforts in the complement of such neighborhood. Section \ref{Sunrise_sec} develops what is really the main insight of our study: a suitable decomposition in replacement of \eqref{20200616_18:28}, inspired in the classical sunrise lemma in harmonic analysis. Finally, Section \ref{Sect_Proof} brings the proof itself, in which we put together all the pieces in our board, and conclude by carefully analyzing a dichotomy that naturally arises.  

\smallskip

Once the work in Sections \ref{Sec_Preli} to \ref{Sect_Proof} is complete, and we are able to fully see the strategy working in the model case of Theorem \ref{Thm1}, we take a moment in Section \ref{Sec_strategy}
to reflect on what really are the abstract core elements that make the method work. In fact, the reach of our sunrise strategy goes way beyond the situation of Theorem \ref{Thm1}, and these abstract guidelines pave the way for further applications that we now describe.

\subsection{Further applications} \label{Sec_Further_app_Intro}
\subsubsection{Hardy-Littlewood maximal operator on the sphere} \label{Section_polar}Let  $\mathbb{S}^{d} \subset \R^{d+1}$ be the unit sphere and let $d(\zeta,\omega)$ denote the geodesic distance between two points $\zeta,\omega \in \mathbb{S}^{d}$. Let $\mathcal{B}_r(\zeta) \subset \mathbb{S}^{d}$ be the open geodesic ball of center $\zeta \in \mathbb{S}^{d}$ and radius $r >0$, that is $\mathcal{B}_r(\zeta) = \{ \omega \in \mathbb{S}^{d} \ : \ d(\zeta,\omega) < r\}$. For $f \in L^1(\mathbb{S}^{d})$ we define the uncentered Hardy-Littlewood maximal function $\widetilde{\mathcal{M}}f$ by
$$\widetilde{\mathcal{M}}f(\xi) = \sup_{\{\mathcal{B}_r(\zeta) \ : \ \xi \in \mathcal{B}_r(\zeta)\}} \frac{1}{\sigma(\mathcal{B}_r(\zeta))}\int_{\mathcal{B}_r(\zeta)} |f(\omega)|\,\d \sigma(\omega),$$
where $\sigma = \sigma_d$ denotes the usual surface measure on the sphere $\mathbb{S}^{d}$. The centered version ${\mathcal{M}}$ would be defined with centered geodesic balls. Fix ${\bf e} = (1, 0,0,\ldots,0) \in \R^{d+1}$ to be the north pole. We say that a function $f: \mathbb{S}^{d} \to \C$ is {\it polar} if for every $\xi, \zeta \in \mathbb{S}^{d}$ with $ \xi \cdot {\bf e}  =  \zeta \cdot {\bf e} $ we have $f(\xi) = f(\zeta)$. This is the analogue, in the spherical setting, of a radial function in the Euclidean setting. Let $W^{1,1}_{\rm pol}(\mathbb{S}^d)$ be the subspace of  $W^{1,1}(\mathbb{S}^d)$ consisting of polar functions. 

\smallskip

For $d =1$ and $f \in W^{1,1}(\mathbb{S}^1)$ (not necessarily polar), and for $d \geq 2$ and $f \in W^{1,1}_{\rm pol}(\mathbb{S}^d)$, we have that $\widetilde{\mathcal{M}}f$ is weakly differentiable and 
\begin{equation}\label{20200720_12:56}
\|\nabla \widetilde{\mathcal{M}}f\|_{L^1(\mathbb{S}^{d})} \lesssim_d \|\nabla f\|_{L^1(\mathbb{S}^{d})}.
\end{equation}
The case $d=1$ follows by an adaptation of the ideas of Tanaka \cite{Ta} and Aldaz and P\'{e}rez L\'{a}zaro \cite{AP} to the periodic setting (in fact, in dimension $d=1$ the inequality holds with constant $C=1$, i.e. the total variation does not increase). The case $d \geq 2$ is subtler and was established in \cite[Theorem 2]{CGR}.  Complementing \eqref{20200720_12:56} we establish the following.

\begin{theorem}\label{Thm2}
The map $f \mapsto  \nabla \widetilde{\mathcal{M}}f$ is continuous from $W^{1,1}(\mathbb{S}^1)$ to $L^1(\mathbb{S}^1)$ and from $W^{1,1}_{\rm pol}(\mathbb{S}^d)$ to $L^1(\mathbb{S}^d)$ for $d\geq 2$.
\end{theorem}
The proof of this result is given in \S\ref{Section_polar}.

\subsubsection{Non-tangential Hardy-Littlewood maximal operator} For $\alpha \geq 0$ and $f \in L^1_{\rm loc}(\R)$ we define the non-tangential Hardy-Littlewood maximal operator $M^{\alpha}$ by 
\begin{equation}\label{20200721_13:39}
M^{\alpha}f(x) = \sup_{|x-y| \leq \alpha r} \frac{1}{2r} \int_{y-r}^{y+r} |f(t)|\, \d t.
\end{equation}
With our previous notation, note that when $\alpha = 0$ we have $M^0 = M$ (the centered Hardy-Littlewood maximal operator) and when $\alpha =1$ we have $M^1 = \widetilde{M}$ (the uncentered one). In \cite{Ra}, J. P. Ramos established a beautiful regularity result for such operators: for $\alpha \geq \frac{1}{3}$ and $f: \R \to \R$ of bounded variation, one has 
\begin{equation}\label{20200720_13:11}
{\rm Var}(M^{\alpha}f) \leq {\rm Var}(f),
\end{equation}
where ${\rm Var}(f)$ denotes the total variation of the function $f$. The interesting feature of \eqref{20200720_13:11} is the variation contractivity property (i.e. the constant $C=1$ on the right-hand side of the inequality). The case $\alpha = 1$ had been previously established by Aldaz and J. P\'{e}rez L\'{a}zaro in \cite{AP}. For $0 \leq \alpha < \frac{1}{3}$ inequality \eqref{20200720_13:11} holds with a constant $C$ that is no larger that $240004$ due to the work of Kurka \cite{Ku} (and it is currently unknown if one can bring down this constant $C$ to $1$). The mechanism that implies the contractivity in \eqref{20200720_13:11} is the fact that $M^{\alpha}f$ has no local maxima in the disconnecting set (say, with $f$ slightly smoother, and then one approximates). The threshold $\alpha =\frac{1}{3}$ is geometrically relevant for this absence of local maxima, and we will review later how it comes into play. From \eqref{20200720_13:11} one can show that when $\alpha \geq \frac{1}{3}$ and $f \in W^{1,1}(\R)$ then $M^{\alpha}f$ is weakly differentiable and 
\begin{equation*}
\|(M^{\alpha}f)'\|_{L^1(\R)} \leq \|f'\|_{L^1(\R)}.
\end{equation*}

\smallskip

We now consider an extension of this operator to several variables. Let $\mathcal{Q}$ be the family of all closed cubes in $\R^d$ (with any possible center and any possible orientation, not necessarily with sides parallel to the original axes). If $Q \in \mathcal{Q}$ we let $\alpha Q$ be the cube that is the dilation of $Q$ by a factor $\alpha$ with the same center. For $f \in L^1_{\rm loc}(\R^d)$ we now define
\begin{equation}\label{20200721_13:40}
M^{\alpha}f(x) = \sup_{x \in \alpha Q} \ \intav{Q} |f(y)|\,\d y.
\end{equation}
Note that in dimension $d =1$ definitions \eqref{20200721_13:39} and \eqref{20200721_13:40} agree. We establish here the following result.

\begin{theorem}\label{Thm_non_tang_HL}
Let $\alpha \geq \frac{1}{3}$ and $M^{\alpha}$ be defined by \eqref{20200721_13:40}.
\begin{enumerate}
\item[(i)] If $d =1$ the map $f \mapsto (M^{\alpha}f)'$ is continuous from $W^{1,1}(\R)$ to $L^1(\R)$.
\smallskip
\item[(ii)] If $d\geq 2$ and $f \in W^{1,1}_{\rm rad}(\R^d)$ then $M^{\alpha}f$ is weakly differentiable. Moreover, the map $f \mapsto \nabla M^{\alpha}f$ is bounded and continuous from $W^{1,1}_{\rm rad}(\R^d)$ to $L^1(\R^d)$.
\end{enumerate}
\end{theorem}
The proof of this result is given in \S\ref{Non_tang_HL}. The boundedness in Theorem \ref{Thm_non_tang_HL} (ii) is also a novelty in the theory. We give a self-contained argument that, {\it en passant}, provides an alternative approach  to \cite{Ra} in order to prove \eqref{20200720_13:11}; see Proposition \ref{Prop_new_arg_cubes} for details.

\subsubsection{Non-tangential heat flow maximal operator} For $t>0$ and $x \in \R^d$ let
$$\varphi_t(x) = \frac{1}{(4\pi t)^{d/2}} \, e^{-|x|^2/4t}$$
be the heat kernel. For $\alpha \geq 0$, consider the following maximal operator
\begin{equation}\label{20200727_11:41}
M_{\varphi}^{\alpha}f(x) = \sup_{t >0 \, ; \, |y-x|\leq \alpha \sqrt{t}} \,(|f| * \varphi_t)(y).
\end{equation}
If we write
$$u(x,t) := (|f| * \varphi_t)(x)$$
then we know that $u$ verifies the heat equation $u_t - \Delta u = 0$ in $\R^d \times (0,\infty)$ with $\lim_{t\to 0^+} u(x,t) = |f(x)|$ for a.e. $x \in \R^d$ (provided $f$ has some minimal regularity, say $f \in L^p(\R^d)$ for any $1 \leq p \leq \infty$). In this sense, when $\alpha =0$, $M_{\varphi}^{0}f(x)$ is just the $\sup$ of $u(x,t)$ over the vertical fiber over $x$ (the heat flow maximal operator) and, when $\alpha >0$, $M_{\varphi}^{\alpha}f(x)$ is a sup of $u(y,t)$ within a parabolic region with lower vertex in $x$ (the non-tangential heat flow maximal operator). The regularity of maximal operators associated to partial differential equations was studied in \cite{BOS, CFS, CGR, CS}. In particular, if $f \in W^{1,1}(\R)$ or $f \in W_{\rm rad}^{1,1}(\R^d)$, if $d \geq 2$, then $M_{\varphi}^0f$ is weakly differentiable and 
$$\|\nabla M_{\varphi}^0 f\|_{L^1(\mathbb{R}^{d})} \lesssim_d \|\nabla f\|_{L^1(\mathbb{R}^{d})}.$$
This is proved in dimension $d=1$ in \cite[Theorem 1]{CS} (in fact with constant $C=1$, and we have the variation contractivity property), and in dimension $d\geq 2$ in \cite[Theorem 1]{CGR}. A key idea in the proof of this inequality is the fact that $M_\varphi^0 f$ is a subharmonic function in the disconnecting set (say, if $f$ continuous and lies in some $L^p(\R^d)$ for $1 \leq p < \infty$), and in particular there are no strict local maxima in such set. Here we consider the non-tangential case and prove the following.

\begin{theorem}\label{Thm_non_tang_heat_flow}
Let $\alpha >0 $ and $M^{\alpha}_\varphi$ defined by \eqref{20200727_11:41}. The map $f \mapsto \nabla M^{\alpha}_\varphi f$ is bounded and continuous from $W^{1,1}(\R)$ to $L^1(\R)$ and from $W^{1,1}_{\rm rad}(\R^d)$ to $L^1(\R^d)$ for $d \geq 2$.
\end{theorem}
The proof of this result is given in \S\ref{Proof_heat_flow}. The boundedness part will follow from the circle of ideas in \cite{CGR}, and the main novelty here is the continuity part that will follow from our sunrise strategy. The continuity in the centered case $\alpha = 0$ is not exactly currently accessible with our methods, and we comment a bit on the difficulties for this and other operators of convolution type (e.g. with the Poisson kernel) in \S \ref{Conc_rem}.

\enlargethispage{0.5\baselineskip}

\subsection{A word on notation} We write $A \lesssim_d B$ or $A = O_d(B)$ if $A \leq C B$ for a certain constant $C>0 $ that may depend on the dimension $d$. We say that $A  \simeq_d B$ if $A \lesssim_d B$ and $B \lesssim_d A$. If there are other parameters of dependence, they will also be indicated. Variables like $x,y,z$ will generally be reserved for $\R^d$, while variables like $r,s,t,u,a,b,\ell, \tau, \rho,\eta, \theta, \delta, \varepsilon$ will generally be reserved for $\R$. The surface area of the sphere $\mathbb{S}^{d-1} \subset \R^d$ is denoted by $\omega_{d-1}$. The open ball $B_r(0)$ will be simply called $B_r$. We assume that our functions are real-valued (or $\pm \infty)$. The characteristic function of a set $E$ is $\chi_E$.

\section{Preliminaries: regularity and convergence}\label{Sec_Preli}

\subsection{Basic regularity} \label{Sec_basic_reg}

Let us first make some generic considerations about radial functions in $\R^d$ and weak derivatives. Let $f:\R^d \to \R \cup \{\pm \infty\}$ be a radial function. With a (hopefully) harmless abuse of notation, throughout the text we write $f(x)$ when we referring to this function in $\R^d$, and $f(r)$ when referring to its radial restriction in $(0,\infty)$, where $r = |x|$. 
\smallskip

A radial function $f(x)$ is weakly differentiable in $\R^d \setminus \{0\}$ if and only if its radial restriction $f(r)$ is weakly differentiable in $(0,\infty)$. In this case, the weak gradient $\nabla f$ of $f(x)$ and the weak derivative $f'$ of $f(r)$ are related by $\nabla f(x) = f'(|x|)\frac{x}{|x|}$ (see \cite[Lemma 4]{CGR} for details). Hence $f(x) \in W^{1,1}_{\rm rad}(\R^d)$ if and only if $f(r) \in W^{1,1}((0,\infty), r^{d-1}\d r)$, and 
\begin{equation}\label{20200701_11:44}
\int_{\R^d} |\nabla f(x)|\,\d x  = \omega_{d-1} \int_0^{\infty} |f'(r)|\,r^{d-1}\,\d r <\infty.
\end{equation}
In particular, after a possible redefinition on a set of measure zero, one can take $f(r)$ continuous in $(0,\infty)$; in fact, absolutely continuous in each interval $[\delta, \infty) \subset (0,\infty)$, and hence differentiable a.e. in $(0,\infty)$. This is equivalent to saying that $f(x)$ is continuous in $\R^d\setminus\{0\}$ and differentiable a.e. in $\R^d\setminus\{0\}$. {\it It is henceforth agreed that we will always work under such regularity assumptions}. Note that this is essentially the best regularity one can expect, since at the origin our function $f \in W^{1,1}_{\rm rad}(\R^d)$ may have a singularity like $|x|^{\alpha}$ with $-d+1 < \alpha < 0$.

\smallskip

If $f \in W^{1,1}_{\rm rad}(\R^d)$ is continuous in $\R^d\setminus \{0\}$, it is not hard to show that $\wM f$ is also continuous in $\R^d\setminus \{0\}$ (and, of course, radial). From \cite{Lu2} we know that $\wM f$ is weakly differentiable in $\R^d$ and 
\begin{equation}\label{20200702_12:24}
\|\nabla \wM f\|_{L^1(\R^d)} \lesssim_d \, \|\nabla f\|_{L^1(\R^d)}.
\end{equation}
As in \eqref{20200701_11:44}, it follows that $\wM f(r)$ is absolutely continuous in each interval $[\delta, \infty) \subset (0,\infty)$, and hence differentiable a.e. in $(0,\infty)$. Observe that both $f$ and $\wM f$ vanish at infinity (recall that $\wM f \in L^{1,\infty}(\R^d)$). In fact, a bit more can be said. Since
$$ \wM f(r) = - \int_r^{\infty} \big(\wM f\big)'\!(t)\, \d t,$$ 
we have
\begin{align}\label{20200706_10:01}
\begin{split}
(d-1) \int_0^{\infty} \wM f(r)\, r^{d-2}\, \d r & = (d-1) \int_0^{\infty}\left( \int_r^{\infty} - \big(\wM f\big)'\!(t)\, \d t\right) r^{d-2}\, \d r  \\
& \leq (d-1) \int_0^{\infty} \left(\int_r^{\infty}  \big|\big(\wM f\big)'\!(t)\big|\, \d t\right) r^{d-2}\, \d r  \\
& =  (d-1) \int_0^{\infty} \int_0^{t}  r^{d-2} \, \big|\big(\wM f\big)'\!(t)\big|\, \d r \,\d t \\
& = \int_0^{\infty}  \big|\big(\wM f\big)'\!(t)\big|\, t^{d-1} \,\d t < \infty.
\end{split}
\end{align}
The latter is finite from \eqref{20200702_12:24}. An analogous computation holds with $|f(r)|$ replacing $\wM f(r)$. Hence $r \mapsto |f(r)|\,r^{d-1}$ and $r \mapsto \wM f(r)\,r^{d-1}$ have integrable derivatives in $(0,\infty)$, and by the fundamental theorem of calculus the limits $\lim_{r \to \infty} |f(r)|\,r^{d-1}$ and $\lim_{r \to \infty} \wM f(r)\,r^{d-1}$ must exist. If these limits were not zero, one would plainly contradict \eqref{20200706_10:01} and therefore
\begin{equation}\label{20200706_10:24}
\lim_{r \to \infty} |f(r)|\,r^{d-1} = \lim_{r \to \infty} \wM f(r)\,r^{d-1} = 0.
\end{equation}
Another application of the fundamental theorem of calculus shows that the limits $\lim_{r \to 0^+} |f(r)|\,r^{d-1}$ and $ \lim_{r \to 0^+} \wM f(r)\,r^{d-1}$ must also exist. If these were not zero, one would contradict the fact that $f$ and $\wM f$ belong to $L^{d/(d-1)}(\R^d)$ (the former by Sobolev embedding, and the latter by the boundedness of $\wM$ in $L^{d/(d-1)}(\R^d)$). Hence
\begin{equation}\label{20200708_11:57}
\lim_{r \to 0^+} |f(r)|\,r^{d-1} = \lim_{r \to 0^+} \wM f(r)\,r^{d-1} = 0.
\end{equation}

\subsection{Splitting and non-negative functions} \label{Split_Sec}The following result will be very useful in our strategy. We state it here in a more general version, having in mind the additional applications given in the forthcoming Section \ref{Further_app}.

\begin{lemma}[Divide and conquer]\label{reduction to lateral}
Let $I \subset \R$ be an open interval and let $\mu$ be a non-negative measure on $I$ such that $\mu$ and the Lebesgue measure are mutually absolutely continuous. Let $\mc{X}$ be the space of functions $\psi:I \to \R$ satisfying the following conditions: 
\begin{enumerate}
\item[(i)] $\psi$ is absolutely continuous in each compact interval of $I$;
\item[(ii)] $\psi' \in L^1(I, \d\mu)$. 
\end{enumerate}
Let $h$ and $g$ be two functions in $\mc{X}$ and let $\{h_{j}\}_{j\geq 1}$ and $\{g_{j}\}_{j\geq 1}$ be two sequences in $\mc{X}$ such that 
\begin{enumerate}
\item[(a)] $h_j(x) \to h(x)$ and $g_j(x) \to g(x)$ as $j \to \infty$, for all $x \in I$;
\item[(b)] $\|h'_{j}-h'\|_{L^1(I, \d\mu)}\to 0$ and $\|g'_{j}-g'\|_{L^1(I, \d\mu)}\to 0$ as $j\to\infty$.
\end{enumerate}
Define $f_{j}:=\max\{g_{j},h_{j}\}$ for each $j \geq 1$ and $f:=\max\{g,h\}$. Then  $f \in \mc{X}$, $\{f_{j}\}_{j\geq 1} \subset \mc{X}$, and 
$$
\|f'_{j}-f'\|_{L^1(I, \d\mu)}\to 0 \ \ \ {\rm as} \ \ \ j\to\infty.
$$
\end{lemma}
\begin{proof} This is essentially \cite[Lemma 11]{CMP}, with  minor modifications in the proof.
\end{proof}
\noindent{\sc Remark:} For the proof of Theorem \ref{Thm1} we shall use Lemma \ref{reduction to lateral} with $I = (0,\infty)$ and $\d\mu(r) = r^{d-1}\d r$. A basic modification of Lemma \ref{reduction to lateral} allows us to also consider the situation where $I = \mathbb{S}^1$ and $\mu$ is the arclength measure. This shall be used in Section \ref{Further_app}.

\smallskip

We are then able to perform a basic reduction.

\begin{proposition}[Reduction to non-negative functions]\label{pass_modulus}
Let $f\in W^{1,1}_{\rm rad}(\R^d)$ and $\{f_{j}\}_{j \geq 1}\subset W^{1,1}_{\rm rad}(\R^d)$ be such that $\|f_{j}-f\|_{W^{1,1}(\R^d)}\to0$ as $j\to\infty$. Then $\||f_{j}|-|f|\|_{W^{1,1}(\R^d)}\to0$ as $j\to\infty$.
\end{proposition}
\begin{proof}
Since $\big||f_{j}|-|f|\big| \leq |f_j -f|$ pointwise, it follows that $\||f_{j}|-|f|\|_{L^{1}(\R^d)}\to0$ as $j \to \infty$. By the fundamental theorem of calculus, for each $r \geq \delta$, 
\begin{equation}\label{20200721_10:13}
|f(r) - f_j(r)| = \left|\int_r^{\infty} (f' - f_j')(t) \,\d t \right|\lesssim_{\delta} \int_{0}^{\infty} \big|(f' - f_j')(t)\big|\, t^{d-1} \,\d t \ \to \ 0,
\end{equation}
as $j \to \infty$. Noting that $|f| = \max\{f, -f\}$, the fact that $\|\nabla|f_{j}|-\nabla|f|\|_{L^{1}(\R^d)} = \omega_{d-1} \||f_{j}|'-|f|'\|_{L^1((0,\infty), \,r^{d-1}\d r)}\to 0$ follows directly from Lemma \ref{reduction to lateral}.
\end{proof}

Since the maximal operator only sees the absolute value of a function, in light on Proposition \ref{pass_modulus} {\it we can assume for the rest of the proof of Theorem \ref{Thm1} that all the functions considered are non-negative}.

\subsection{Connecting and disconnecting sets, and local maxima} \label{Sec2.3_CD} Let $f \in W^{1,1}_{\rm rad}(\R^d)$ be continuous in $\R^d\setminus \{0\}$ and non-negative. Define the $d$-dimensional disconnecting set by
\begin{equation*}
\mathcal{D}(f) = \big\{x \in \R^d\setminus \{0\}\ : \ \wM f(x) > f(x)\big\},
\end{equation*}
and its corresponding one-dimensional radial version
\begin{equation*}
D(f) = \{ |x|\ : \ x \in \mathcal{D}(f)\}.
\end{equation*}
Analogously, we define the connecting set 
\begin{equation*}
\mathcal{C}(f) = \big\{x \in \R^d\setminus \{0\}\ : \ \wM f(x) = f(x)\big\},
\end{equation*}
and its one-dimensional radial version
\begin{equation*}
C(f) = \{ |x|\ : \ x \in \mathcal{C}(f)\}.
\end{equation*}
Note that the sets $\mathcal{D}(f) \subset \R^d\setminus \{0\}$ and $D(f) \subset (0, \infty)$ are open. Note also that if $r \in C(f)$ is a point of differentiability of $f$, then we must have $f'(r) = 0$; otherwise one could find a small ball over which the average beats $f(r)$, and $r$ would belong to $D(f)$ instead. We now recall a basic result of the theory, that will be crucial for our sunrise construction later in Section \ref{Sunrise_sec}.

\begin{proposition}[Absence of local maxima]\label{No_local_maxima}
Let $f \in W^{1,1}_{\rm rad}(\R^d)$. The function $\wM f(r)$ does not have a strict local maximum in $D(f)$.
\end{proposition}

\begin{proof}
By a strict local maximum we mean a point $r_0 \in D(f)$ for which there exist $s_0$ and $t_0$ with $s_0 < r_0 < t_0$, $[s_0,t_0] \subset D(f)$, such that $\wM f(r) \leq \wM f(r_0)$ for all $r \in [s_0,t_0]$ and $\wM f(s_0), \wM f(t_0) < \wM f(r_0)$. Let $x_0 \in \R^d$ be such that $|x_0| = r_0$, and consider a closed ball $\overline{B}$ such that $x_0 \in \overline{B}$ and $\wM f(x_0) = \intav{\,\overline{B}} f$ (observe that such a ball exists and has a strictly positive radius since $x_0 \in \mc{D}(f))$. From the above we see that $\{|y| \ : \ y \in \overline{B}\} \subset (s_0, t_0)$. Since $[s_0, t_0] \subset D(f)$ we obtain
\begin{equation*}
\wM f(x_0) = \intav{\,\overline{B}} f \ < \ \intav{\,\overline{B}} \wM f \leq \wM f(x_0)\,,
\end{equation*}
a contradiction.
\end{proof}

\subsection{Pointwise convergence} \label{Point_conv_sec}For $x \in \R^d \setminus \{0\}$, let us define ${\mc B}(f;x)$ as the set of closed balls $\overline{B}$ that realize the supremum in the definition of the maximal function at the point $x$, that is 
\begin{equation}\label{20200619_10:30}
{\mc B}(f;x) = \left\{\overline{B}\,;   \ x \in \overline{B} \ : \   \wM f(x) = \, \intav{\, \overline{B}} \, f(y)\,\d y \right\}.
\end{equation}
Note that we include possibility that $\overline{B} = \{x\}$ (we may think of radius zero here), with the understanding that $\intav{\{x\}}  f(y)\,\d y := f(x)$. Therefore, we note that ${\mc B}(f;x)$ is always non-empty. The next proposition qualitatively describes the derivative of the maximal function.

\begin{proposition}[The derivative of the maximal function]\label{Prop_derivative_inside}
Let $f \in W^{1,1}_{\rm rad}(\mathbb{R}^{d})$ be a non-negative function and let $x \in \R^d \setminus \{0\}$ be a point of differentiability of $\wM f$. Then, for any ball $\overline{B} \in {\mc B}(f;x)$ of strictly positive radius, we have
$$\nabla \wM f(x) = \, \intav{\,\overline{B}} \nabla f(y)\,\d y.$$
\end{proposition}
\begin{proof}
This is contained in \cite[Lemma 2.2]{Lu2}. 
\end{proof}
This leads us to our considerations on pointwise convergence issues.

\begin{proposition}[Pointwise convergence for $\wM$]\label{Prop_pointwise_conv} Let $f\in W^{1,1}_{\rm rad}(\R^d)$ and $\{f_{j}\}_{j \geq 1}\subset W^{1,1}_{\rm rad}(\R^d)$ be such that $\|f_{j}-f\|_{W^{1,1}(\R^d)}\to0$ as $j\to\infty$. The following statements hold.

\begin{enumerate}
\item[(i)] For each $\delta >0$, we have $f_j(r) \to f(r)$ and $\wM f_j(r) \to \wM f(r)$ uniformly in the set $\{r \geq \delta\}$ as $j \to \infty$.

\smallskip

\item[(ii)] If $x \in \R^d \setminus \{0\}$,  $\overline{B_{s_j}(z_j)} \in  {\mc B}(f_j;x)$\footnote{Recall that we allow for the possibility $\overline{B_{0}(x)} = \{x\}.$} and $(s,z) \in [0,\infty) \times \R^d$ is an accumulation point of the sequence $\{(s_j, z_j)\}_{j \geq 1}$, then $\overline{B_{s}(z)} \in  {\mc B}(f;x)$.

\smallskip

\item[(iii)] For almost all $r \in D(f)$ we have $\big(\wM f_j\big)'(r) \to \big(\wM f\big)'(r)$ as $j \to \infty$.
\end{enumerate}
\end{proposition}

\begin{proof} Part (i). The uniform convergence $f_j(r) \to f(r)$ as $ j \to \infty$ follows from \eqref{20200721_10:13}. Using the sublinearity of $\wM$ we also have
\begin{align}\label{20200724_11:39}
\begin{split}
\big| \wM f(r) -  \wM f_j(r)\big| \leq \wM (f - f_j)(r) &= - \int_r^{\infty} \big(\wM(f - f_j)\big)'(t) \,\d t \lesssim_{\delta} \int_0^{\infty} \big|\big(\wM(f - f_j)\big)'(t)\big|\,t^{d-1} \,\d t\\
& \lesssim_{\delta, d} \int_{0}^{\infty} \big|(f' - f_j')(t)\big|\, t^{d-1} \,\d t \ \to \ 0
\end{split}
\end{align}
as $j \to \infty$. Note the use of \eqref{20200702_12:24} in the last passage above.

\smallskip

\noindent Part (ii). This follows by using part (i). One may divide in the cases $s >0$ and $s=0$.

\smallskip

\noindent Part (iii). Assume that $D(f) \subset (0,\infty)$ has positive measure, otherwise we are done (in particular we may assume that $f \not\equiv 0$). Let $E(f) \subset (0,\infty)$ (resp. $E(f_j) \subset (0,\infty)$) be the set of measure zero where $\wM f(r)$ (resp. $\wM f_j(r)$) is not differentiable. Let us prove the statement for any $r \in D(f) \setminus \big(E(f) \cup \big(\cup_{j=1}^{\infty} E(f_j))\big)$.

\smallskip

Let $x \in \R^d \setminus \{0\}$ be such that $|x| = r$. Then $\wM f$ and all $\big\{\wM f_j\big\}_{j\geq 1}$ are differentiable at $x$. From part (i) we find that $x \in \mathcal{D}(f_j)$ for $j \geq j_0$. Using parts (i) and (ii), and the fact that $\{\|f_j\|_{L^1(\R^d)}\}_{j\geq 1}$ is bounded we find that there exist $\varepsilon >0$ , $N >0$ and $j_1 \geq j_0$ such that if $\overline{B_{s_j}(z_j)} \in  {\mc B}(f_j;x)$ for $j \geq j_1$ then $\varepsilon \leq s_j \leq N$. The result now follows from part (ii) and Proposition \ref{Prop_derivative_inside}.
\end{proof}

\section{Control near the origin}\label{Sec_control}

In this section we develop the first part of our overall strategy of the proof of Theorem \ref{Thm1}, by establishing a control of the convergence near the origin. This is inspired in an argument of \cite{GR}.

\begin{proposition}[Control near the origin] \label{Prop_control_origin}Let $f\in W^{1,1}_{\rm rad}(\R^d)$ and $\{f_{j}\}_{j \geq 1}\subset W^{1,1}_{\rm rad}(\R^d)$ be such that $\|f_{j}-f\|_{W^{1,1}(\R^d)}\to0$ as $j\to\infty$. Then for every $\varepsilon >0$ there exists $\eta = \eta(\varepsilon) > 0$ such that 
\begin{equation*}
\int_{B_{\eta}} \big| \nabla \wM f\big| < \varepsilon \ \ \ {\rm and} \ \ \ \int_{B_{\eta}} \big| \nabla \wM f_j\big| < \varepsilon
\end{equation*}
for all $j \geq j_1(\varepsilon, \eta)$.
\end{proposition}
\begin{proof} If $f =0$ the result follows directly from \eqref{20200702_12:24}. So let us assume that $f \not\equiv 0$. Recall that we may assume that all our functions are non-negative. For a generic $g \in W^{1,1}_{\rm rad}(\R^d)$ non-negative we claim that for any $\eta>0$ and $\ell >2$ we have
\begin{equation}\label{20200708_12:16}
\int_{B_{\eta}} \big| \nabla \wM g\big| \lesssim_d \int_{B_{\ell \eta}} \big| \nabla  g\big| + \frac{1}{\ell^d} \int_{\R^d} \big| \nabla  g\big| + g(\ell \eta) ( \ell \eta)^{d-1}.
\end{equation} 
The conclusion of Proposition \ref{Prop_control_origin} plainly follows from this claim by taking $\ell$ large, $\eta$ small (with the product $\ell \eta$ still small), and using \eqref{20200708_11:57} and the fact that $f_j(\ell\eta)$ converges pointwise to $f(\ell\eta)$ given by Proposition \ref{Prop_pointwise_conv} (i).

\smallskip

Let us then prove the claim \eqref{20200708_12:16}. For each $x \in \R^d \setminus\{0\}$ let $r_x$ be the maximal radius of a closed ball in ${\mc B}(g;x)$. Define the set 
$$\mathcal{A}:=\Big\{x\in B_{\eta} \setminus \{0\} \ : \ r_x\ge \frac{\ell \eta}{4}\Big\}.$$
Using Proposition \ref{Prop_derivative_inside} we find that 
\begin{equation}\label{20200708_13:50}
\int_{\mathcal{A}} \big| \nabla \wM g\big| \lesssim_d \int_{B_{\eta}} \frac{\big\|\nabla  g \big\|_{L^1(\R^d)}
}{ (\ell \eta)^d}  \lesssim_d \frac{\big\|\nabla g \big\|_{L^1(\R^d)}}{ \ell^d}.
\end{equation}
We now take care of the integral over $B_{\eta} \setminus \mathcal{A}$. For every $\beta >0$ define a function $g_{\beta} \in W^{1,1}_{\rm rad}(\R^d)$ by 
\begin{equation*}
g_{\beta}(r) = \left\{
\begin{array}{lcl}
g(r) & {\rm for} & 0< r < \ell \eta;\\
\frac{-g(\ell \eta)}{\beta}r+\frac{(\ell \eta+\beta)g(\ell\eta)}{\beta}  & {\rm for} & \ell \eta\leq r \leq \ell \eta + \beta;\\
0 & {\rm for} & \ell \eta + \beta< r.
\end{array}
\right.
\end{equation*}
Assume for a moment that $\ell \eta$ is a point of differentiability of $g(r)$. Then, for $\beta$ small enough, we have that $g_{\beta} \leq g$, and hence $\wM g_{\beta} \leq \wM g$. If $x \in B_{\eta} \setminus \mathcal{A}$, then $r_x < \ell \eta/4$ and any ball $\overline{B} \in {\mc B}(g;x)$ will be entirely contained in $\overline{B_{\eta +\frac{\ell \eta}{2}}} \subset \overline{B_{\ell \eta}}$. This implies that $\wM g(x) \leq \wM g_{\beta}(x)$ for such $x$, and hence $\wM g_{\beta} = \wM g$ in the set $B_{\eta} \setminus \mathcal{A}$ (note also that this set is open by Proposition \ref{Prop_pointwise_conv} (ii)). Using \eqref{20200702_12:24} we then find
\begin{align*}
\int_{B_{\eta} \setminus \mathcal{A}} \big| \nabla \wM g\big| &= \int_{B_{\eta} \setminus \mathcal{A}} \big| \nabla \wM g_{\beta}\big| \leq \int_{\R^d} \big| \nabla \wM g_{\beta}\big| \lesssim_d \int_{\R^d} |\nabla g_{\beta}| \\
& = \int_{B_{\ell \eta}} \big| \nabla  g\big| + \omega_{d-1} \frac{g(\ell \eta)}{\beta} \int_{\ell \eta}^{\ell \eta+\beta} t^{d-1}\,\d t.
\end{align*}
Sending $\beta \to 0$ we obtain 
\begin{equation}\label{20200708_13:51}
\int_{B_{\eta} \setminus \mathcal{A}} \big| \nabla \wM g\big| \lesssim_d \int_{B_{\ell \eta}} \big| \nabla  g\big| + \omega_{d-1} \,g(\ell \eta) \,( \ell \eta)^{d-1}.
\end{equation}
By adding \eqref{20200708_13:50} and \eqref{20200708_13:51} we arrive at \eqref{20200708_12:16}. For any fixed $\eta >0$, the right-hand side of \eqref{20200708_12:16} is continuous in $\ell$, and hence the inequality holds also if $\ell \eta$ is not a point of differentiability of $g(r)$.
\end{proof}

\section{The sunrise construction}\label{Sunrise_sec}

The purpose of this section is to present a decomposition that will play the role of \eqref{20200616_18:28} in our multidimensional radial case, and understand its basic properties.

\subsection{Definition} Let $f \in W^{1,1}_{\rm rad}(\R^d)$ be continuous in $\R^d\setminus \{0\}$ and non-negative. From \eqref{20200706_10:24} we henceforth denote $f(+\infty) = \wM f(+\infty) := 0$. For technical reasons that will become clearer later (e.g. see Proposition \ref{Prop_pointwise_conv_lateral} below), it will be convenient to avoid a neighborhood of the origin in our discussion, and we let $\rho >0$ be a fixed parameter throughout this section. {\it It should be clear from the start that all the new constructions in this section depend on such parameter $\rho >0$,} and we shall excuse ourselves from an explicit mention to it in some of the passages and definitions below in order to simplify the notation.

\smallskip

We start by decomposing the open set $D(f) \cap (\rho,\infty)$ into a countable union of open intervals 
\begin{equation}\label{20200720_14:54}
D(f) \cap (\rho, \infty) = \bigcup_{i=1}^{\infty} \big(a_i(f;\rho), b_i(f;\rho)\big).
\end{equation}
When the dependence on $f$ and $\rho$ is clear, we shall simply write $(a_i,b_i)$ instead of $\big(a_i(f;\rho), b_i(f;\rho)\big)$. Let $(a_i,b_i)$ be a generic interval of this decomposition. Proposition \ref{No_local_maxima}  guarantees the existence of $\tau_i^- =  \tau_i^-(f;\rho)$ and $\tau_i^+ =  \tau_i^+(f;\rho)$ such that $a_i \leq  \tau_i^- \leq  \tau_i^+ \leq b_i$ and 
\begin{equation*}
[ \tau_i^-, \tau_i^+ ] = \big\{ r \in [a_i,b_i] \ : \ \wM f(r) = \min\big\{ \wM f(s) \ ; \  s \in [a_i,b_i]\big\}\big\}.
\end{equation*}
That is, $[ \tau_i^-, \tau_i^+ ]$ is the interval of points of minima of $\wM f$ in $[a_i,b_i]$. Note that possibilities like $\tau_i^- = \tau_i^+$, $\tau_i^- =a_i = \rho$ or $\tau_i^+ = b_i =+\infty$ are all duly accounted for. From Proposition \ref{No_local_maxima} we know that $\wM f(r)$ is non-increasing in $[a_i, \tau_i^-]$ and non-decreasing in $[\tau_i^+, b_i]$.

\smallskip

Inspired by the classical construction of the sunrise lemma in harmonic analysis we now consider the following functions. For $r \in (a_i, \tau_i^-)$ (this interval may be empty) define
\begin{equation}\label{20200711_14:12}
W^i_R f(r) = \max\left\{ \max_{r \leq t \leq \tau_i^-}f(t) \ , \ \wM f(\tau_i^-)\right\},
\end{equation}
and for $r \in (\tau_i^+, b_i)$ (this interval may be empty) define
\begin{equation*}
W^i_L f(r) = \max\left\{ \max_{\tau_i^+ \leq t \leq r}f(t) \ , \ \wM f(\tau_i^+)\right\}.
\end{equation*}
We are now in position to define our analogues of the lateral maximal functions in \eqref{20200616_18:28}. For each $r \in (\rho,\infty)$ we define the functions $\wM_R f = \wM_R(f;\rho)$ and $\wM_L f = \wM_L(f;\rho)$ at the point $r$ by
\begin{equation*}
\wM_R f(r)= \left\{
\begin{array}{lcl}
\wM f(r) & {\rm if} & r \in C(f) \ {\rm or} \ r \in [\tau_i^-, b_i) \ {\rm for \ some}\ i \geq 1.\\
W^i_R f(r)& {\rm if} & r \in (a_i, \tau_i^-)\ {\rm for \ some}\ i \geq 1;
\end{array}
\right. 
\end{equation*}
and
\begin{equation*}
\wM_L f (r)= \left\{
\begin{array}{lcl}
\wM f(r) & {\rm if} & r \in C(f) \ {\rm or} \ r \in (a_i, \tau_i^+] \ {\rm for \ some}\ i \geq 1;\\
W^i_L f(r)& {\rm if} & r \in (\tau_i^+, b_i)\ {\rm for \ some}\ i \geq 1.
\end{array}
\right. 
\end{equation*}

\smallskip

\noindent {\sc Remark:} Note that we are not defining these functions in the interval $(0,\rho]$.

\begin{figure}
  \centering
  \includegraphics[height=7.5cm]{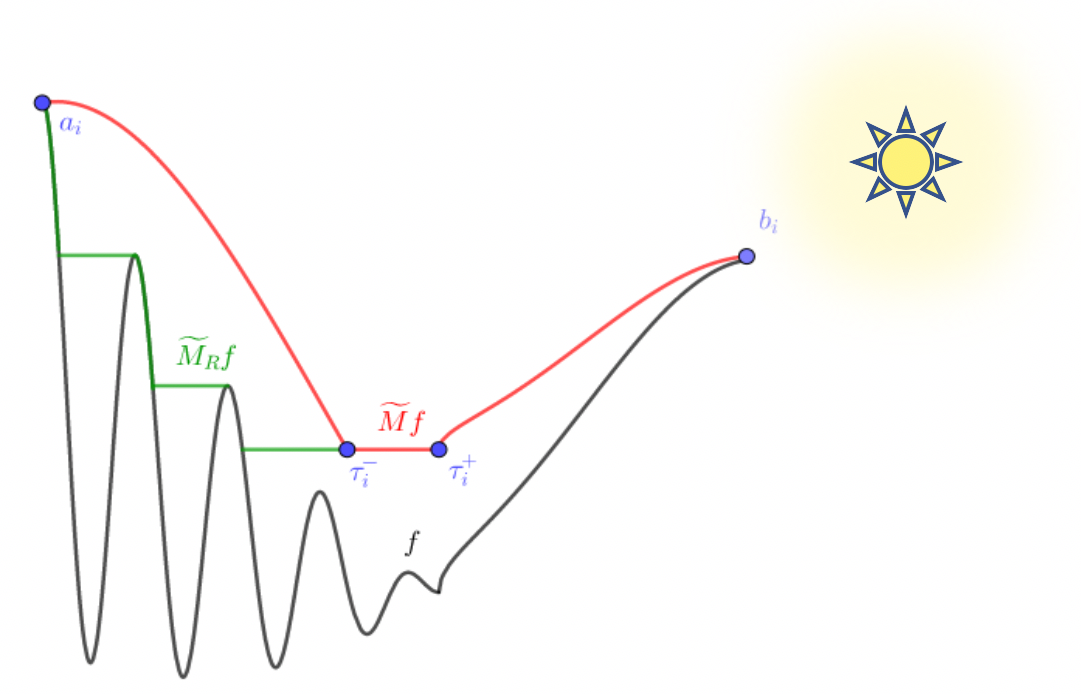} 
    \caption{The sunrise lateral maximal function $\wM_R f$ in a disconnecting interval $(a_i,b_i)$.}
\label{fig:Sunrise_pic}
\end{figure}

\smallskip

Before moving on to discuss the basic properties of these new functions, let us point out two important facts. First, in dimension $d=1$ it is not necessarily true that $\wM_Rf = M_Rf$ and $\wM_L f = M_Lf$ in the interval $(\rho, \infty)$, where $M_R$ and $M_L$ are the classical one-sided maximal operators, to the right and left, respectively (consider, for instance, $f$ being two sharp bumps to the right of $\rho$). Second, note that $\wM_Rf$ and $\wM_Lf$ are generated from $f$ indirectly, i.e. passing through $\wM f$, and it is not in principle true that the operators $f \mapsto \wM_Rf$ and $f \mapsto \wM_Lf$ are sublinear. This is a source of technical difficulty in the proof, especially in the upcoming Proposition \ref{Prop_pointwise_conv_lateral}, that will be carefully handled.

\subsection{Basic properties} \label{Sec4_basic}From the definition, for all $r \in (\rho,\infty)$ one plainly sees that  
\begin{equation}\label{20200707_13:14}
f(r) \leq \wM_Rf(r) \leq \wM f(r) \ \ {\rm and} \ \ f(r) \leq \wM_Lf(r) \leq \wM f(r),\end{equation}
and
\begin{equation}\label{20201713_09:38}
\wM f(r) = \max\Big\{\wM_Rf(r)\, , \, \wM_Lf(r)\Big\}.
\end{equation}
Also, for any $\rho < r<s < \infty$, one can show that 
\begin{equation*}
\left| \wM_R f(r) - \wM_R f(s)\right| \leq \int_r^s |f'(t)|\,\d t + \int_r^s \big|(\wM f)'(t)\big|\,\d t \,,
\end{equation*}
and the same holds for $\wM_L f$. For this one may consider the different cases when $r$ and $s$ belong to $C(f)$ or $D(f)$. This plainly implies that $\wM_R f$ and $\wM_L f$ are absolutely continuous in $(\rho,\infty)$. In particular, $\wM_R f$ and $\wM_L f$ are differentiable a.e. in $(\rho,\infty)$.

\smallskip

As before, let us define the disconnecting set $D_R(f) = D_R(f;\rho)$ and the connecting set $C_R(f) = C_R(f;\rho)$ by
\begin{equation}\label{20200713_09:46}
D_R(f) = \big\{r \in (\rho,\infty) \, : \, \wM_Rf(r) > f(r)\big\} \ \ {\rm and} \ \ C_R(f) = \big\{r \in (\rho,\infty) \, : \, \wM_R f(r) = f(r)\big\},
\end{equation}
and, analogously, we define $D_L(f) = D_L(f;\rho)$ and $C_L(f) = C_L(f;\rho)$ by
$$D_L(f) = \big\{r \in (\rho,\infty) \ : \ \wM_L f(r) > f(r)\big\}  \ \ {\rm and} \ \ C_L(f) = \big\{r \in (\rho,\infty) \ : \ \wM_L(f)(r) = f(r)\big\}.$$
We now prove a fundamental property of our construction.
\begin{proposition}[Monotonicity]\label{Monotonicity_lem}
The following monotonicity properties hold:
$$\big(\wM_Rf\big)'(r) \geq 0 \ \ {\rm a.e. \ in}\  D_R(f)\ \ \ {\rm and} \ \ \ \big(\wM_R f\big)'(r) \leq 0 \ \ {\rm a.e. \ in} \ C_R(f),$$
and 
\begin{equation*}\big(\wM_L f\big)'(r) \leq 0 \ \ {\rm a.e. \ in}\  D_L(f)\ \ \ {\rm and} \ \ \ \big(\wM_L f\big)'(r) \geq 0 \ \ {\rm a.e. \ in} \ C_L(f).
\end{equation*}
\end{proposition}
\begin{proof} We consider $\wM_Rf$. The proof for $\wM_Lf$ is essentially analogous. Let us consider the disjoint decomposition
\begin{equation}\label{20200707_13:07}
D(f) \cap (\rho, \infty) = D^-(f) \cup D^0(f) \cup D^+(f)\,,
\end{equation}
where $D^-(f) = D^-(f;\rho)$, $D^0(f) = D^0(f;\rho)$ and $D^+(f) = D^+(f;\rho)$ are defined by
\begin{equation}\label{20200707_13:08}
D^-(f) = \bigcup_{i=1}^{\infty} (a_i, \tau_i^-) \ ; \ D^0(f) = \bigcup_{i=1}^{\infty} \Big([\tau_i^-, \tau_i^+] \cap D(f) \cap (\rho, \infty) \Big) \ \ {\rm and} \ \ D^+(f) = \bigcup_{i=1}^{\infty} (\tau_i^+, b_i).
\end{equation}
Note that $\big( D^0(f) \cup D^+(f)\big) \subset D_R(f) \subset D(f)$ and hence
\begin{equation}\label{20200711_23:20}
D_R(f) = D^0(f) \cup D^+(f) \cup \big(D_R(f) \cap D^-(f)\big).
\end{equation}
Also,
$$C_R(f) = \big(C(f) \cap (\rho, \infty)\big) \cup \big(C_R(f) \cap D^-(f)\big).$$
We claim that the derivative of $\wM_R f$ in $(\rho, \infty)$ is given by
\begin{equation}\label{20200707_11:57}
\big(\wM_R f\big)'(r) = \left\{
\begin{array}{ll}
 \big(\wM f\big)'(r)\geq 0 & {\rm for \ a.e.} \ r \in D^+(f);\\
 \big(\wM f\big)'(r) = 0 & {\rm for \ a.e.} \ r \in D^0(f);\\
 0 & {\rm for \ all} \ r \in D_R(f) \cap D^-(f);\\
  f'(r) = 0& {\rm for \ a.e.} \ r \in C(f) \cap (\rho, \infty);\\
 f'(r) \leq 0 & {\rm for \ a.e.} \ r \in C_R(f) \cap D^-(f).
\end{array}
\right.
\end{equation}

Let us look at the disconnecting set first. Since $\wM_Rf(r) = \wM f(r)$ is non-decreasing in each $(\tau_i^+, b_i)$, we find that $\big(\wM_Rf\big)'(r) = \big(\wM f\big)'(r)\geq 0$ a.e. in $D^+(f)$. In each point $r \in (\tau_i^-, \tau_i^+)$ (if this set is non-empty) we have $\wM_Rf = \wM f$ being constant in a neighborhood of $r$, and hence $\big(\wM f\big)'(r) = 0$. If $r \in D_R(f) \cap D^-(f)$, then $\wM_Rf(r)$ is also constant in a neighborhood of $r$, and we have $\big(\wM_Rf\big)'(r) = 0$.

\smallskip

As for the connecting set, if $r \in C(f) \cap (\rho,\infty)$ is a point of differentiability of $\wM_Rf$, $\wM f$ and $f$, and is not an isolated point of $C(f) \cap (\rho,\infty)$ (note that this is still a.e. in $C(f) \cap (\rho,\infty)$), we observe that $\big(\wM_R f\big)'(r) = \big(\wM f\big)'(r) = f'(r)= 0$; see the discussion in \S \ref{Sec2.3_CD}. We are left with analyzing $C_R(f) \cap D^-(f)$. Note that $W^i_Rf$ is non-increasing in $\big(a_i, \tau_i^-\big)$, which means that $\big(\wM_R f\big)'(r) = \big(W^i_Rf\big)'(r) \leq 0$ a.e. in $(a_i, \tau_i^-)$ for each $i \geq 1$, and hence for a.e. $r \in C_R(f) \cap D^-(f)$. Then, if $r \in C_R(f) \cap D^-(f)$ is a point of differentiability of $\wM_Rf$ and $f$, and is not an isolated point of $C_R(f) \cap D^-(f)$ (which is still a.e. in $C_R(f) \cap D^-(f)$) we have $\big(\wM_R f\big)'(r) = f'(r) \leq 0$. 
\end{proof}
\noindent {\sc Remark}: From the description \eqref{20200707_11:57} note that $\big(\wM_R f\big)' \in  L^1((\rho,\infty), r^{d-1}\d r)$, and so does $\big(\wM_L f\big)'$.

\subsection{Pointwise convergence} We now move to a crucial and delicate result in our strategy, the analogue of Proposition \ref{Prop_pointwise_conv} for the lateral operators $\wM_R$ and $\wM_L$. Note how the use of the sublinearity of $\wM$ allows for a relatively simple proof of Proposition \ref{Prop_pointwise_conv} (i). Unfortunately, sublinearity is a tool we do not possess here, and we must handle the situation differently. Our approach will be more of a {\it tour-de-force} one, in which we carefully study the many different building blocks and possibilities of the sunrise construction. We will split the content now into two propositions, as the proofs will be more elaborate. Recall that we assume that all functions considered here are non-negative.

\begin{proposition}[Pointwise convergence for $\wM_R$ and $\wM_L$]\label{Prop_pointwise_conv_lateral} Let $f\in W^{1,1}_{\rm rad}(\R^d)$ and $\{f_{j}\}_{j \geq 1}\subset W^{1,1}_{\rm rad}(\R^d)$ be such that $\|f_{j}-f\|_{W^{1,1}(\R^d)}\to0$ as $j\to\infty$. Then, for each $r \in (\rho,\infty)$, we have $\wM_R f_j(r) \to \wM_R f(r)$ and $\wM_L f_j(r) \to \wM_L f(r)$ as $j \to \infty$.
\end{proposition}
\begin{proof}
Let us prove the statement for $\wM_R$. The proof for $\wM_L$ is essentially analogous. Recall the decomposition given by \eqref{20200707_13:07} - \eqref{20200707_13:08}. Given $\varepsilon >0$, from Proposition \ref{Prop_pointwise_conv} (i) there exists $j_0 = j_0(\varepsilon)$ such that 
\begin{equation}\label{20200712_10:09}
|f_j(t) - f(t)| \leq \varepsilon \ \  {\rm and} \ \  \big|\wM f_j(t) - \wM f(t)\big| \leq \varepsilon
\end{equation} 
for all $j \geq j_0$ and all $t \in (\rho, \infty)$. Any mention of $j_0(\varepsilon)$ below refers to this uniform convergence. We divide our analysis into the following exhaustive list of cases.

\medskip

\noindent {\it Case} 1: $r \in C(f)$. In this case $\wM_Rf(r) = \wM f(r) = f(r)$. From Proposition \ref{Prop_pointwise_conv} (i) we know that $\wM f_j(r) \to \wM f(r)$ and that $f_j(r) \to f(r)$ as $j \to \infty$. The desired result follows from \eqref{20200707_13:14}.

\medskip

\noindent {\it Case} 2: $r \in D^+(f)$. In this case $r \in (\tau_i^+, b_i)$ for some $i \geq 1$, and we know that $\wM_Rf(r) = \wM f(r) > \max\big\{f(r), \wM f\big(\tau_i^+\big)\big\}$. Let $s$ be such that $\tau_i^+ < s < r$ and $\wM f(s) < \wM f(r)$. Then $[s,r] \subset D^+(f)$ and by Proposition \ref{Prop_pointwise_conv} (i) we have that $[s,r] \subset D(f_j)$ and $\wM f_j(s) < \wM f_j(r)$ for $j \geq j_1$. This plainly implies that $r \in D^+(f_j)$ and hence $\wM_Rf_j(r) = \wM f_j(r)$ for $j \geq j_1$. The result follows from another application of Proposition \ref{Prop_pointwise_conv} (i).

\medskip

\noindent {\it Case} 3: $r \in D^0(f)$. In this case $r \in [\tau_i^-, \tau_i^+] \cap D(f) \cap (\rho,\infty)$ for some $i \geq 1$ and we have $\wM_Rf(r) = \wM f(r) >f(r)$. In particular note that $f \not\equiv0$. Hence, we cannot have $b_i = +\infty$ since this would plainly imply $\tau_i^- = \tau_i^+ = b_i = +\infty$, contradicting our situation. We then have two subcases to consider:

\medskip

\noindent {\it Subcase} 3.1: $\tau_i^+ < b_i < +\infty$. Given $\varepsilon >0$ sufficiently small, let $\tau_i^+ <u < b_i$ be such that 
$$\wM f(r) \geq \max\big\{f(t) \ : \ t \in [r,u]\big\} + 3\varepsilon.$$
Then $[r,u] \subset D(f)$ and by Proposition \ref{Prop_pointwise_conv} (i) we have that $[r,u] \subset D(f_j)$ and $\wM f_j(u) > \wM f_j(r)$ for $j \geq j_1 \geq j_0(\varepsilon)$. We now observe two possibilities for each $j \geq j_1$: 
\begin{enumerate}
\item[(i)] if $r \in D^0(f_j) \cup D^+(f_j)$ we have that $\wM_Rf_j(r) = \wM f_j(r)$ and hence
$$\big|\wM_R f_j(r) - \wM_R f(r)\big| = \big|\wM f_j(r) - \wM f(r)\big| \leq \varepsilon.$$
\item[(ii)] if $r \in D^-(f_j)$ then, by the considerations above, the corresponding left minimum $\tau_{i_j}^-(f_j)= \tau_{i_j}^-(f_j;\rho)$ of $\wM f_j$ in the disconnecting open interval of $D(f_j) \cap (\rho,\infty)$ that contains $r$ is such that $r < \tau_{i_j}^-(f_j) < u$ and $\wM_Rf_j(r) = W^{i_j}_R f_j(r) =\wM f_j(\tau_{i_j}^-(f_j))$. In this situation we have
\begin{equation*}
\wM f(r) + \varepsilon \geq \wM f_j(r) \geq \wM f_j(\tau_{i_j}^-(f_j)) \geq \wM f(\tau_{i_j}^-(f_j)) - \varepsilon \geq \wM f(r)  - \varepsilon,
\end{equation*}
which implies that 
\begin{align*}
\big|\wM_R f_j(r) - \wM_R f(r)\big| & = \big|\wM f_j(\tau_{i_j}^-(f_j)) - \wM f(r)\big| \leq \varepsilon.
\end{align*}
\end{enumerate}

\medskip

\noindent {\it Subcase} 3.2: $\tau_i^+ = b_i < +\infty$. Let $\varepsilon >0$ be given. From Proposition \ref{Prop_pointwise_conv} (i) we have that $r \in D(f_j)$ for $j \geq j_1 \geq j_0(\varepsilon)$. We now observe three possibilities for each $j \geq j_1$:
\begin{enumerate}
\item[(i)] if $r \in D^0(f_j) \cup D^+(f_j)$ we have that $\wM_Rf_j(r) = \wM f_j(r)$ and hence
$$\big|\wM_R f_j(r) - \wM_R f(r)\big| = \big|\wM f_j(r) - \wM f(r)\big| \leq \varepsilon.$$
\item[(ii)] if $r \in D^-(f_j)$ and the corresponding left minimum $\tau_{i_j}^-(f_j)= \tau_{i_j}^-(f_j;\rho)$ is such that $r < \tau_{i_j}^-(f_j) \leq b_i$ we have 
$$\wM f(r) + \varepsilon \geq \wM f_j(r) \geq \wM_Rf_j(r) \geq \wM f_j(\tau_{i_j}^-(f_j)) \geq \wM f(\tau_{i_j}^-(f_j)) - \varepsilon = \wM f(r) - \varepsilon,$$
from which we conclude that 
\begin{equation*}
\big|\wM_R f_j(r) - \wM_R f(r)\big|  = \big|\wM_R f_j(r) - \wM f(r)\big| \leq \varepsilon.
\end{equation*}
\item[(iii)] if $r \in D^-(f_j)$ and the corresponding left minimum $\tau_{i_j}^-(f_j)= \tau_{i_j}^-(f_j;\rho)$ is such that $b_i < \tau_{i_j}^-(f_j)$ we have (recall that $\wM f(b_i) = f(b_i)$ in this situation)
$$\wM f(r) + \varepsilon \geq \wM f_j(r) \geq \wM_Rf_j(r) \geq f_j(b_i) \geq f(b_i) -\varepsilon  = \wM f(b_i) - \varepsilon = \wM f(r) - \varepsilon,$$
and again we conclude that 
\begin{equation*}
\big|\wM_R f_j(r) - \wM_R f(r)\big|  = \big|\wM_R f_j(r) - \wM f(r)\big| \leq \varepsilon.
\end{equation*}
\end{enumerate}

\medskip

\noindent {\it Case} 4: $r \in D^-(f)$. In this case $r \in (a_i, \tau_i^-)$ for some $i \geq 1$ and we have $\wM f(r) > \wM_Rf(r) = W^i_R f(r)$ defined in \eqref{20200711_14:12}. In particular $f \not\equiv 0$. We consider the following subcases:

\medskip

\noindent {\it Subcase} 4.1: $W^i_R f(r) = \max_{r \leq t \leq \tau_i^- }f(t) \geq \wM f(\tau_i^-)$. Given $\varepsilon >0$ sufficiently small, let $s$ be such that $r \leq s < \tau_i^-$ and $0 \leq W^i_R f(r) - f(s) \leq \varepsilon$. Let $u$ and $v$ be such that $r \leq  s < u < v < \tau_i^-$ and 
$$\min\{ \wM f(r), \wM f(\tau_i^-) + \varepsilon\} > \wM f(u) > \wM f(v)> \wM f(\tau_i^-).$$
Then $[r,v] \subset D^-(f)$ and by Proposition \ref{Prop_pointwise_conv} (i) we have that $[r,v] \subset D(f_j)$ and $\wM f_j(r) > \wM f_j(u) > \wM f_j(v)$ for $j \geq j_1 \geq j_0(\varepsilon)$. This implies that $r \in D^{-}(f_j)$ for $j \geq j_1$, and we again let $\tau_{i_j}^-(f_j)= \tau_{i_j}^-(f_j;\rho)$ be the corresponding left minimum. Observe that $u < \tau_{i_j}^-(f_j)$. From this we get 
\begin{align}\label{20200711_15:27}
\wM f_j(\tau_{i_j}^-(f_j)) < \wM f_j(u) \leq \wM f(u) + \varepsilon \leq \wM f(\tau_i^-) + 2\varepsilon \leq W^i_R f(r)  + 2\varepsilon,
\end{align}
and using \eqref{20200711_15:27} we also get
\begin{align}\label{20200711_23:04}
\begin{split}
 \max_{r \leq t \leq\tau_{i_j}^-(f_j)}f_j(t) &\leq  \max\left\{\max_{r \leq t \leq u}f_j(t)\,,  \max_{u \leq t \leq \tau_{i_j}^-(f_j) }\wM f_j(t) \right\}  \\
 & \leq \max\left\{ \max_{r \leq t \leq u}f(t) + \varepsilon\,, \wM f_j(u) \right\} \leq W^i_R f(r)  + 2\varepsilon.
 \end{split}
 \end{align}
From \eqref{20200711_15:27} and \eqref{20200711_23:04} we have, for $j \geq j_1$, 
\begin{align*}
W^i_R f(r) + 2\varepsilon \geq   W^{i_j}_R f_j(r) \geq f_j(s) \geq f(s) - \varepsilon \geq W^i_R f(r) - 2\varepsilon\,,
\end{align*}
which implies 
\begin{align*}
\big|\wM_R f_j(r) - \wM_R f(r)\big|  = \big|W^{i_j}_R f_j(r) - W^i_R f(r) \big| \leq 2\varepsilon.
\end{align*}

\medskip

\noindent {\it Subcase} 4.2: $W^i_R f(r) = \wM f(\tau_i^-) > \max_{r \leq t \leq \tau_i^- }f(t)$. Note that $b_i < +\infty$, otherwise we would have $\tau_i^- = \tau_i^+ = b_i = +\infty$, contradicting our situation. We analyze here the two possibilities:

\medskip

\noindent  \S 4.2.1: $\tau_i^+ < b_i < +\infty$. Given $\varepsilon >0$ sufficiently small, let $u$ be such that $r < \tau_i^- \leq \tau_i^+ < u < b_i$ and 
$$\wM f(\tau_i^-) \geq \max\big\{f(t) \ : \ t \in [r,u]\big\} + 3\varepsilon.$$
Then $[r,u] \subset D(f)$ and by Proposition \ref{Prop_pointwise_conv} (i) we have that $[r,u] \subset D(f_j)$, $\wM f_j(r) > \wM f_j(\tau_i^-)$ and $\wM f_j(\tau_i^-) < \wM f_j(u)$ for $j \geq j_1 \geq j_0(\varepsilon)$. This implies that $r \in D^{-}(f_j)$ for $j \geq j_1$ and the corresponding left minimum $\tau_{i_j}^-(f_j)= \tau_{i_j}^-(f_j;\rho)$ is such that $r < \tau_{i_j}^-(f_j) < u$. In this scenario, note that $\wM_Rf_j(r) = W^{i_j}_R f_j(r) =\wM f_j(\tau_{i_j}^-(f_j))$ and, for $j \geq j_1$, 
\begin{align*}
\wM f(\tau_{i}^-) +\varepsilon  \geq \wM f_j(\tau_{i}^-)  \geq  \wM f_j(\tau_{i_j}^-(f_j)) \geq \wM f(\tau_{i_j}^-(f_j)) - \varepsilon \geq \wM f(\tau_i^-) - \varepsilon,
\end{align*}
which implies
\begin{align*}
\big|\wM_R f_j(r) - \wM_R f(r)\big|  = \big|\wM f_j(\tau_{i_j}^-(f_j)) - \wM f(\tau_{i}^-)\big| \leq \varepsilon.
\end{align*}

\medskip

\noindent  \S 4.2.2: $\tau_i^+ = b_i < +\infty$. Given $\varepsilon >0$ sufficiently small, let $u$ and $v$ be such that $r < u < v < \tau_i^-$ and 
$$\min\{ \wM f(r), \wM f(\tau_i^-) + \varepsilon\} > \wM f(u) > \wM f(v)> \wM f(\tau_i^-).$$
Then $[r,v] \subset D^-(f)$ and by Proposition \ref{Prop_pointwise_conv} (i) we have that $[r,v] \subset D(f_j)$ and $\wM f_j(r) > \wM f_j(u) > \wM f_j(v)$ for $j \geq j_1 \geq j_0(\varepsilon)$. This implies that $r \in D^{-}(f_j)$ for $j \geq j_1$, and we again let $\tau_{i_j}^-(f_j)= \tau_{i_j}^-(f_j;\rho)$ be the corresponding left minimum. Observe that $u < \tau_{i_j}^-(f_j)$ and hence
\begin{align}\label{20200711_23:07}
\wM f_j(\tau_{i_j}^-(f_j)) < \wM f_j(u) \leq \wM f(u) + \varepsilon \leq \wM f(\tau_i^-) + 2\varepsilon.
\end{align}
Using \eqref{20200711_23:07} we get
\begin{align}\label{20200711_23:08}
\begin{split}
 \max_{r \leq t \leq\tau_{i_j}^-(f_j)}f_j(t) &\leq  \max\left\{\max_{r \leq t \leq u}f_j(t)\,,  \max_{u \leq t \leq \tau_{i_j}^-(f_j) }\wM f_j(t) \right\} \\ &\leq \max\left\{ \wM f(\tau_i^-) + \varepsilon\,, \wM f_j(u) \right\} 
  \leq\wM f(\tau_i^-) + 2\varepsilon.
  \end{split}
 \end{align}
From \eqref{20200711_23:07} and \eqref{20200711_23:08} we conclude that 
\begin{equation}\label{20200711_23:01}
W^{i_j}_R f_j(r) \leq \wM f(\tau_i^-) + 2\varepsilon.
\end{equation}
For the other inequality we proceed as follows. If $\tau_{i_j}^-(f_j) \leq b_i$ we have
$$W^{i_j}_R f_j(r) \geq \wM f_j(\tau_{i_j}^-(f_j))\geq \wM f(\tau_{i_j}^-(f_j)) - \varepsilon \geq \wM f(\tau_i^-) - \varepsilon.$$
If $\tau_{i_j}^-(f_j) > b_i$ we have  (recall that $\wM f(b_i) = f(b_i)$ in this situation)
$$W^{i_j}_R f_j(r) \geq  f_j(b_i) \geq f(b_i) - \varepsilon = \wM f(b_i) - \varepsilon \geq \wM f(\tau_i^-) - \varepsilon.$$
In either case we conclude that
\begin{equation}\label{20200711_23:00}
W^{i_j}_R f_j(r) \geq \wM f(\tau_i^-) - \varepsilon.
\end{equation}
Finally, from \eqref{20200711_23:01} and \eqref{20200711_23:00} we reach the desired conclusion
\begin{align*}
\big|\wM_R f_j(r) - \wM_R f(r)\big|  = \big|W^{i_j}_R f_j(r) - \wM f(\tau_{i}^-)\big| \leq 2\varepsilon.
\end{align*}
This completes the proof.
\end{proof}

\begin{proposition}[Pointwise convergence for the derivatives of $\wM_R$ and $\wM_L$]\label{Prop_pointwise_conv_lateral_derivative} Let $f\in W^{1,1}_{\rm rad}(\R^d)$ and $\{f_{j}\}_{j \geq 1}\subset W^{1,1}_{\rm rad}(\R^d)$ be such that $\|f_{j}-f\|_{W^{1,1}(\R^d)}\to0$ as $j\to\infty$. Then, for almost all $r \in D_R(f)$, we have $\big(\wM_R f_j\big)'(r) \to \big(\wM_R f\big)'(r)$ as $j \to \infty$, and for almost all $r \in D_L(f)$ we have $\big(\wM_L f_j\big)'(r) \to \big(\wM_L f\big)'(r)$ as $j \to \infty$.
\end{proposition}

\begin{proof} We prove the statement for $\wM_R$ as the proof for $\wM_L$ is essentially analogous. For each $\varepsilon >0$ we keep defining $j_0(\varepsilon)$ by \eqref{20200712_10:09}. Recalling decomposition \eqref{20200711_23:20} we divide again our analysis into cases.

\medskip

\noindent {\it Case} 1: $D^+(f)$. Let us consider an interval $(\tau_i^+, b_i)$ for some $i \geq 1$. For each $\tau_i^+ < u < v < b_i$ we may choose $s$ with $\tau_i^+ < s < u < v < b_i$ such that $\wM f(s) < \wM f_j(u)$. Then $[s,v] \subset D^+(f)$ and by Proposition \ref{Prop_pointwise_conv} (i) we have that $[s,v] \subset D(f_j)$ and $\wM f_j(s) < \wM f_j(u)$ for $j \geq j_1$. This plainly implies that $[u,v] \subset D^+(f_j)$ and hence $\wM_Rf_j(r) = \wM f_j(r)$ for all $r \in [u,v]$ and $j \geq j_1$. The result follows from Proposition \ref{Prop_pointwise_conv} (iii).

\medskip

\noindent {\it Case} 2: $D^0(f)$. Since we want to prove the result almost everywhere, it is sufficient to consider only the intervals $[\tau_i^-, \tau_i^+] \cap D(f) \cap (\rho,\infty)$ where $\tau_i^- < \tau_i^+$ (in particular, this implies that $b_i < +\infty$). Let $u$ and $v$ be such that $\tau_i^- < u < v < \tau_i^+$. We consider two subcases:

\medskip

\noindent {\it Subcase} 2.1: $\tau_i^+ < b_i < +\infty$. Given $\varepsilon >0 $ sufficiently small, let $\tau_i^+ <  s < b_i$ be such that 
$$\wM f(u) = \wM f(v)  \geq \max\big\{f(t) \ : \ t \in [u,s]\big\} + 3\varepsilon.$$
From Proposition \ref{Prop_pointwise_conv} (i) we know that  $[u,s] \subset D(f_j)$ and $\wM f_j(s) > \wM f_j(\tau_i^+)$ for $j \geq j_1 \geq j_0(\varepsilon)$. Let $\tau_{i_j}^-(f_j)= \tau_{i_j}^-(f_j;\rho)$ be the corresponding left minimum of $\wM f_j$ in the disconnecting open interval of $D(f_j) \cap (\rho,\infty)$ that contains $[u,s]$. Note that $\tau_{i_j}^-(f_j) < s$ and for $r \in [u,v]$ we have
\begin{align*}
\wM_Rf_j(r) = \left\{
\begin{array}{lcc}
\wM f_j(r) & {\rm if} & \tau_{i_j}^-(f_j) \leq r \leq v;\\
W^{i_j}_R f_j(r) = \wM f_j(\tau_{i_j}^-(f_j))& {\rm if} & u \leq r < \tau_{i_j}^-(f_j).
\end{array}
\right.
\end{align*}
Then, for a.e. $r \in [u,v]$ we have
\begin{align}\label{20200711_00:58}
\begin{split}
\big(\wM_Rf_j\big)'(r) = \left\{
\begin{array}{lcc}
\big(\wM f_j\big)'(r) & {\rm if} & \tau_{i_j}^-(f_j) \leq r \leq v;\\
0& {\rm if} & u \leq r < \tau_{i_j}^-(f_j).
\end{array}
\right.
\end{split}
\end{align}
We conclude from \eqref{20200707_11:57} and Proposition \ref{Prop_pointwise_conv} (iii).

\medskip

\noindent {\it Subcase} 2.2: $\tau_i^+ = b_i < +\infty$. Let $\varepsilon >0$ be sufficiently small so that 
$$\wM f(u) = \wM f(v)  \geq \max\big\{f(t) \ : \ t \in [u,v]\big\} + 3\varepsilon.$$
From Proposition \ref{Prop_pointwise_conv} (i) we know that  $[u,v] \subset D(f_j)$ for $j \geq j_1 \geq j_0(\varepsilon)$, and we again let $\tau_{i_j}^-(f_j)= \tau_{i_j}^-(f_j;\rho)$ be the corresponding left minimum. As before we have, for $r \in [u,v]$, 
\begin{align}\label{20200711_00:49}
\begin{split}
\wM_Rf_j(r) = \left\{
\begin{array}{lcc}
\wM f_j(r) & {\rm if} & \tau_{i_j}^-(f_j) \leq r \leq v;\\
W^{i_j}_R f_j(r) & {\rm if} & u \leq r < \tau_{i_j}^-(f_j).
\end{array}
\right.
\end{split}
\end{align}
Let us take a closer look at the second possibility in \eqref{20200711_00:49}. Observe that if $u \leq r < \tau_{i_j}^-(f_j) \leq b_i$ we have 
$$W^{i_j}_R f_j(r) \geq \wM f_j(\tau_{i_j}^-(f_j)),$$ 
and if $u \leq r < b_i < \tau_{i_j}^-(f_j)$ we have 
$$W^{i_j}_R f_j(r) \geq f_j(b_i) \geq f(b_i) - \varepsilon  = \wM f(u) -  \varepsilon.$$
In either case what matters is that 
\begin{equation*}
W^{i_j}_R f_j(r) = \max\left\{ \max_{r \leq t \leq \tau_{i_j}^-(f_j)}f_j(t) \ , \ \wM f_j(\tau_{i_j}^-(f_j))\right\} = \max\left\{ \max_{v < t \leq \tau_{i_j}^-(f_j)}f_j(t) \ , \ \wM f_j(\tau_{i_j}^-(f_j))\right\} \,,
\end{equation*}
and the expression on the right-hand side is independent of $r$. Then \eqref{20200711_00:49} implies \eqref{20200711_00:58} and we conclude from \eqref{20200707_11:57} and Proposition \ref{Prop_pointwise_conv} (iii) as before.

\medskip

\noindent {\it Case} 3: $D_R(f) \cap D^-(f)$. Let $r \in D_R(f) \cap D^-(f)$. Then $r \in (a_i, \tau_i^-)$ for some $i \geq 1$. Let $s$ be such that $r < s < \tau_i^-$ and $\wM f(r) > \wM f(s)$. Then $[r,s] \subset D^-(f)$ and by Proposition \ref{Prop_pointwise_conv} (i) we have that $[r,s] \subset D(f_j)$ and $\wM f_j(r) > \wM f_j(s)$ for $j \geq j_1$. In particular, this implies that $r \in D^-(f_j)$ for $j \geq j_1$. We have already seen in Proposition \ref{Prop_pointwise_conv_lateral} that if $r \in D_R(f)$ then $r \in D_R(f_j)$ for $j \geq j_2 \geq j_1$. Hence $r \in D_R(f_j) \cap D^-(f_j)$ for $j \geq j_2$ and we conclude by using \eqref{20200707_11:57}.
\end{proof}

\section{The proof}\label{Sect_Proof}

We are now in position to move on to the proof of Theorem \ref{Thm1}. 

\subsection{Setup} Given $f\in W^{1,1}_{\rm rad}(\R^d)$ and $\{f_{j}\}_{j \geq 1}\subset W^{1,1}_{\rm rad}(\R^d)$, all non-negative, and such that $\|f_{j}-f\|_{W^{1,1}(\R^d)}\to0$ as $j\to\infty$, we want to show that 
\begin{equation*}
\big\| \nabla \wM f_j - \nabla \wM f\big\|_{L^1(\R^d)} \to 0 \ \ \ {\rm as}\ \ \ j \to \infty.
\end{equation*}
Given $\varepsilon >0$, let $\eta >0$ be given by Proposition \ref{Prop_control_origin}. Then
\begin{equation*}
\big\| \nabla \wM f_j - \nabla \wM f\big\|_{L^1(B_{\eta})} < 2 \varepsilon
\end{equation*}
for $j \geq j_1(\varepsilon, \eta)$. It is then enough to prove that 
\begin{equation*}
\big\| \nabla \wM f_j - \nabla \wM f\big\|_{L^1(\R^d \setminus B_{\eta})} \to 0 \ \ \ {\rm as}\ \ \ j \to \infty,
\end{equation*}
which is equivalent to 
\begin{equation}\label{20200713_09:35}
\big\| \big(\wM f_j\big)' -  \big(\wM f\big)'\big\|_{L^1((\eta,\infty), \, r^{d-1}\d r)} \to 0 \ \ \ {\rm as}\ \ \ j \to \infty.
\end{equation}
From now on we fix $\rho = \eta /2$ and consider the sunrise construction of the lateral operators $\wM_R$ and $\wM_L$ in Section \ref{Sunrise_sec} with respect to this parameter $\rho$. We have seen in \S \ref{Sec4_basic} that the functions $\wM_Rf$, $\wM_Lf$, $\big\{\wM_Rf_{j}\big\}_{j \geq 1}, \big\{\wM_Lf_{j}\big\}_{j \geq 1}$ are all contained in the space $\mathcal{X}$ of Lemma \ref{reduction to lateral} and hence, by the same lemma and identity \eqref{20201713_09:38}, in order to prove \eqref{20200713_09:35} is is sufficient to show that 
\begin{equation*}
\big\| \big(\wM_R f_j\big)' -  \big(\wM_R f\big)'\big\|_{L^1((\eta,\infty), \, r^{d-1}\d r)} \to 0 \ \  \ {\rm and } \ \  \  \big\| \big(\wM_L f_j\big)' -  \big(\wM_L f\big)'\big\|_{L^1((\eta,\infty), \, r^{d-1}\d r)} \to 0
\end{equation*}
as $j \to \infty$. This is what we are going to do in the remaining of this section. We shall prove it for $\wM_R$ and the proof for $\wM_L$ is essentially analogous.

\subsection{Splitting into the connecting and disconnecting sets} Recall definition \eqref{20200713_09:46}. For the rest of the section let us adopt a simple notation by writing 
$$D = D_R(f) \cap (\eta, \infty) \ ;\  D_j = D_R(f_j) \cap (\eta, \infty) \ ; \ C = C_R(f) \cap (\eta, \infty) \  ;  \ C_j = C_R(f_j) \cap (\eta, \infty).$$ 
Also in the spirit of easing the notation, we sometimes omit the argument of the functions in the integrals below when the context is clear (e.g. writing $f'$ for $f'(r)$) and sometimes use the ``little o" notation for limits (i.e. writing $\lambda_j = o(1)$ when $\lim_{j \to \infty} \lambda_j = 0$). We split our original integral into the following four pieces:
\begin{align*}
\int_{\eta}^{\infty} \big| \big(\wM_R f_j\big)' -  \big(\wM_R f\big)'\big| \,r^{d-1}\,\d r &= \int_{C \cap C_j} + \int_{D \cap C_j} + \int_{C \cap D_j} + \int_{D \cap D_j}\\
& =: (I)_j + (II)_j + (III)_j + (IV)_j.
\end{align*}
Our objective is to show that each of these pieces is $o(1)$ as $j \to \infty$ (note that each of these pieces is non-negative). In what follows the reader should have in mind all times the description \eqref{20200707_11:57} for the derivative of $\wM_R f$. Two of the integral pieces above are particularly simple to analyze, and we clear them out first.

\subsubsection{The term $(I)_j$} By our hypotheses we have
\begin{equation*}
(I)_j = \int_{C \cap C_j} \big| \big(\wM_R f_j\big)' -  \big(\wM_R f\big)'\big| \,r^{d-1}\,\d r = \int_{C \cap C_j} | f_j' -  f'| \,r^{d-1}\,\d r = o(1).
\end{equation*}

\subsubsection{The term $(II)_j$} \label{5.2.2_termII} From Proposition \ref{Prop_pointwise_conv_lateral}, if $r \in D$ then $r \in D_j$ for $j$ large, and hence $\chi_{D \cap C_j}(r) \to 0$ as $j \to \infty$. Therefore, by our hypotheses and dominated convergence we have
\begin{align}\label{20200713_23:28}
\begin{split}
(II)_j &= \int_{D \cap C_j} \big| \big(\wM_R f_j\big)' -  \big(\wM_R f\big)'\big| \,r^{d-1}\,\d r = \int_{\eta}^{\infty} \big| f_j' -  \big(\wM_R f\big)'\big| \,\chi_{D \cap C_j}(r) \,r^{d-1}\,\d r\\
& \leq \int_{\eta}^{\infty} \big| f_j' -  f'| \,\chi_{D \cap C_j}(r) \,r^{d-1}\,\d r + \int_{\eta}^{\infty} \Big(| f'| + \big|\big(\wM_R f\big)'\big|\Big)  \,\chi_{D \cap C_j}(r) \,r^{d-1}\,\d r = o(1).
\end{split}
\end{align}

\vspace{0.05cm}

\subsection{Brezis-Lieb reduction and some useful identities} 

\subsubsection{Using the convergence of the derivatives} The {\it raison d'\^{e}tre} of Proposition \ref{Prop_pointwise_conv_lateral_derivative} is to allow for an application of the classical Brezis-Lieb lemma \cite{BL} to conclude that 
\begin{align}\label{20200713_23:25}
(II)_j + (IV)_j = \int_{D} \big| \big(\wM_R f_j\big)' -  \big(\wM_R f\big)'\big| \,r^{d-1}\,\d r \to 0
\end{align}
as $j \to \infty$ if and only if 
\begin{align}\label{20200713_23:26}
\int_{D} \big| \big(\wM_R f_j\big)' \big| \,r^{d-1}\,\d r \to \int_{D} \big| \big(\wM_R f\big)' \big| \,r^{d-1}\,\d r = \int_{D}  \big(\wM_R f\big)'  \,r^{d-1}\,\d r 
\end{align}
as $j \to \infty$. The equality on the right-hand side of \eqref{20200713_23:26} is due to Proposition \ref{Monotonicity_lem}. From Proposition \ref{Prop_pointwise_conv_lateral_derivative} and Fatou's lemma we already have
\begin{equation}\label{20200713_00:00}
 \int_{D}  \big(\wM_R f\big)'  \,r^{d-1}\,\d r  = \int_{D} \big|\big(\wM_R f\big)' \big| \,r^{d-1}\,\d r \leq \liminf_{j \to \infty} \int_{D}  \big| \big(\wM_R f_j\big)' \big| \,r^{d-1}\,\d r.
 \end{equation}
Let us decompose the open set $D \subset (\eta, \infty)$ into a disjoint union of open intervals:
\begin{align}\label{20200713_16:46}
D = \bigcup_{i = 1}^{\infty} (\alpha_i, \beta_i).
\end{align}
We may have one of the left endpoints in \eqref{20200713_16:46} being $\eta$ and, if that is the case, let us agree that $\eta = \alpha_1$. Note that, as in \eqref{20200706_10:01}, we have
\begin{equation}\label{20200714_09:54}
(d-1) \int_{\eta}^{\infty} \wM_R f(r)\, r^{d-2}\, \d r  \leq \int_{\eta}^{\infty}  \big|\big(\wM_R f\big)'\!(t)\big|\, t^{d-1} \,\d t < \infty.
\end{equation}
Recall also \eqref{20200706_10:24}. Using integration by parts (and dominated convergence with \eqref{20200714_09:54} to properly justify the limiting process in the potentially infinite sum) we have
\begin{align} \label{20200713_12:44} 
 \int_{D}  \big(\wM_R f\big)' \,r^{d-1}\,\d r  & = \sum_{i=1}^{\infty} \int_{\alpha_i}^{\beta_i} \big(\wM_R f\big)' \,r^{d-1}\,\d r \nonumber\\
&  = \sum_{i=1}^{\infty} \left( \left( \wM_R f(\beta_i) \,\beta_i^{d-1} - \wM_R f(\alpha_i) \,\alpha_i^{d-1}\right) - (d-1) \int_{\alpha_i}^{\beta_i} \wM_R f \,\,r^{d-2}\,\d r \right) \nonumber\\
& = \gamma(f) + \sum_{i=1}^{\infty} \left( \left( f(\beta_i) \,\beta_i^{d-1} -  f(\alpha_i) \,\alpha_i^{d-1}\right) - (d-1) \int_{\alpha_i}^{\beta_i}  f\, \,r^{d-2}\,\d r \right) \\
&  \ \ \ \ \ \ \ \ \ \ \ \ \ \ \ \ \ \ \ \ \ + (d-1) \int_D \big(f - \wM_R f\big) \,r^{d-2}\,\d \nonumber \\
& = \gamma(f)  + \int_{D}  f'\, \,r^{d-1}\,\d r  + (d-1) \int_D \big(f - \wM_R f\big) \,r^{d-2}\,\d r,\nonumber
\end{align}
where we introduced the term
\begin{equation}\label{20200713_12:59}
\gamma(f) := \left\{
\begin{array}{lc}
f(\eta) \,\eta^{d-1} - \wM_R f(\eta)\,\eta^{d-1}& {\rm if} \ \eta = \alpha_1;\\
0& {\rm otherwise}.
\end{array}
\right.
\end{equation}
Similarly, we may decompose $D_j = \bigcup_{i = 1}^{\infty} (\alpha_i^j, \beta_i^j)$, with the agreement that if $\eta$ is a left endpoint in this decomposition then $\eta = \alpha_1^j$. We define $\gamma(f_j)$ as in \eqref{20200713_12:59} and proceed as in \eqref{20200713_12:44} to find
\begin{align}\label{20200713_13:19}
\begin{split}
\int_{D_j} \ \big(\wM_R f_j\big)'  \,r^{d-1}\,\d r = \gamma(f_j)  + \int_{D_j}  f_j'\, \,r^{d-1}\,\d r  + (d-1) \int_{D_j} \big(f_j - \wM_R f_j\big) \,r^{d-2}\,\d r.
\end{split}
\end{align}
Combining \eqref{20200713_12:44} and \eqref{20200713_13:19} we arrive at the following identity
\begin{equation}\label{20200713_18:30}
\int_{D_j} \big(\wM_R f_j\big)'  \,r^{d-1}\,\d r - \int_{D_j}  f_j'\, \,r^{d-1}\,\d r  = \int_{D}  \big(\wM_R f\big)'  \,r^{d-1}\,\d r - \int_{D}  f'\, \,r^{d-1}\,\d r + \lambda_j,
\end{equation}
where
\begin{equation}\label{20200714_00:17}
\lambda_j :=  \big(\gamma(f_j) - \gamma(f)\big) + \left(\int_{D_j} \big(f_j - \wM_R f_j\big) \,r^{d-2}\,\d r - \int_{D} \big(f - \wM_R f\big) \,r^{d-2}\,\d r  \right).
\end{equation}

\subsubsection{Smallness of the remainder: analysis of $\lambda_j$} We now claim that $\lambda_j$ defined in \eqref{20200714_00:17} verifies
\begin{align}\label{20200714_10:04}
\lambda_j = o(1).
\end{align}
Note first that $\gamma(f_j) \to \gamma(f)$ as $j \to \infty$. This is an immediate consequence of the pointwise convergences $f_j(\eta) \to f(\eta)$ and $\wM_R f_j(\eta)\to \wM_R f(\eta)$ as $j \to \infty$. The second observation is that 
\begin{align}\label{20200713_18:03}
\int_{D_j} \big(f_j - \wM_R f_j\big) \,r^{d-2}\,\d r \to  \int_{D} \big(f - \wM_R f\big) \,r^{d-2}\,\d r.
\end{align}
as $j \to \infty$. This requires some work to verify. Start by writing the difference in the following form
\begin{align}\label{20200713_17:00}
\begin{split}
&\int_{D_j} \big(f_j - \wM_R f_j\big) \,r^{d-2}\,\d r - \int_{D} \big(f - \wM_R f\big) \,r^{d-2}\,\d r \\
&  =  \int_{\eta}^{\infty} \Big(\big(f_j - \wM_R f_j\big) - \big(f - \wM_R f\big)\Big) \,\chi_D(r)\,r^{d-2}\,\d r  +  \int_{\eta}^{\infty}  \big(f_j - \wM_R f_j\big) \,\chi_{C \cap D_j}(r) \,r^{d-2}\,\d r.
\end{split}
\end{align}
Let $N>0$ be large. Using \eqref{20200707_13:14} and the sublinearity of $\wM$, the portion of each of the two integrals on the right-hand side of \eqref{20200713_17:00} evaluated from $N$ to $\infty$ is bounded in absolute value by 
\begin{equation}\label{20200713_17:38}
4\int_N^{\infty} \Big(\wM f + \wM(f- f_j)\Big)\,r^{d-2}\,\d r.
\end{equation}
A computation as in \eqref{20200706_10:01}, together with \eqref{20200702_12:24}, shows that \eqref{20200713_17:38} is bounded by
\begin{align*}
&\lesssim_d \int_{N}^{\infty}  \Big( \big|\big(\wM f\big)'(t)\big|  + \big|\big(\wM(f- f_j)\big)'(t)\big|\Big)\, t^{d-1} \,\d t \\
& \lesssim_d \int_{N}^{\infty}   \big|\big(\wM f\big)'(t)\big| \, t^{d-1} \,\d t + \int_{0}^{\infty} |(f-f_j)'(t)| \, t^{d-1} \,\d t,
\end{align*}
and by our hypotheses this is small if $N$ is large and $j$ is large. In the interval $[\eta, N]$ all the functions $\wM f_j$ are uniformly bounded (by Proposition \ref{Prop_pointwise_conv} (i)). By applying Proposition \ref{Prop_pointwise_conv} (i), Proposition \ref{Prop_pointwise_conv_lateral} and dominated convergence, we find that the portion of each of the two integrals on the right-hand side of \eqref{20200713_17:00} evaluated from $\eta$ to $N$ converges to zero. This establishes \eqref{20200713_18:03} and hence \eqref{20200714_10:04}.

\subsubsection{Final preparation} We need yet another useful identity to run our upcoming dichotomy scheme. Using Proposition \ref{Monotonicity_lem} (multiple times), and identity \eqref{20200713_18:30} - \eqref{20200714_10:04} (in the third line below), we have
\begin{align}\label{20200713_23:22}
 &\!\! \int_{D}  \big| \big(\wM_R f_j\big)' \big| \,r^{d-1}\,\d r =\!\! \int_{D \cap D_j} \! \big| \big(\wM_R f_j\big)' \big| \,r^{d-1}\d r  + \!\! \int_{D \cap C_j} \! \big| \big(\wM_R f_j\big)' \big| \,r^{d-1}\d r + (III)_j - \!(III)_j\nonumber\\
& = \int_{D \cap D_j}  \big(\wM_R f_j\big)'  \,r^{d-1}\,\d r - \int_{D \cap C_j} f_j' \,\,r^{d-1}\,\d r  + \int_{C \cap D_j}  \Big(\big(\wM_R f_j\big)' -   f'\Big) \,r^{d-1}\,\d r - (III)_j\nonumber\\
& =  \int_{D_j}  \big(\wM_R f_j\big)'  \,r^{d-1}\,\d r -  \int_{D \cap C_j} f_j' \,\,r^{d-1}\,\d r - \int_{C \cap D_j}   f' \,r^{d-1}\,\d r - (III)_j\nonumber\\
& = \int_{D}  \big(\wM_R f\big)'  \,r^{d-1}\,\d r - \int_{D}  f' \,r^{d-1}\,\d r + \int_{D_j}  f_j' \,r^{d-1}\,\d r + o(1)\\
&  \ \ \ \ \ \  \ \ \ \ \ \ \ \ \ \ \ \ \ \ -  \int_{D \cap C_j} f_j' \,\,r^{d-1}\,\d r - \int_{C \cap D_j}   f' \,r^{d-1}\,\d r - (III)_j\nonumber\\
& = \int_{D}  \big(\wM_R f\big)' \,r^{d-1}\,\d r + \int_{D_j} (f_j' - f') \,\,r^{d-1}\,\d r - \int_{D\cap C_j}  (f_j' + f') \,r^{d-1}\,\d r  - (III)_j + o(1)\nonumber\\
& = \int_{D} \big(\wM_R f\big)'  \,r^{d-1}\,\d r - (III)_j + o(1).\nonumber
\end{align}
Note that in the last passage above we used the fact that $\chi_{D \cap C_j}(r) \to 0$ and dominated convergence as in \eqref{20200713_23:28}.

\subsection{Finale: the dichotomy} Let us take a closer look at identity \eqref{20200713_18:30}. For each $j \geq 1$ we have the following dichotomy: either
\begin{equation}\label{20200713_23:00}
\int_{C \cap D_j} \big(\wM_R f_j\big)'  \,r^{d-1}\,\d r \leq \int_{C \cap D_j}  f_j'\, \,r^{d-1}\,\d r 
\end{equation}
or 
\begin{equation}\label{20200713_23:31}
\int_{D \cap D_j} \big(\wM_R f_j\big)' \,r^{d-1}\,\d r \leq \int_{D \cap D_j}  f_j'\, \,r^{d-1}\,\d r + \int_{D} \big(\wM_R f\big)'  \,r^{d-1}\,\d r - \int_{D}  f'\, \,r^{d-1}\,\d r + \lambda_j.
\end{equation}

\subsubsection{Case 1} Assume that we go over the subsequence of $j$’s such that \eqref{20200713_23:00} holds. Using Proposition \ref{Monotonicity_lem} and \eqref{20200713_23:00} we get
\begin{align*}
(III)_j = \int_{C \cap D_j}  \Big(\big(\wM_R f_j\big)' -   f'\Big) \,r^{d-1}\,\d r \leq \int_{C \cap D_j}  \big(f_j' -   f'\big) \,r^{d-1}\,\d r = o(1).
\end{align*}
Then, from \eqref{20200713_23:25}, \eqref{20200713_23:26} and \eqref{20200713_23:22} we find that 
$$(II)_j + (IV)_j = o(1).$$
Then $(IV)_j = o(1)$ and the proof is complete in this case.

\subsubsection{Case 2} Assume now that we go over the subsequence of $j$’s such that \eqref{20200713_23:31} holds. Using Proposition \ref{Monotonicity_lem}, \eqref{20200714_10:04} and \eqref{20200713_23:31}, we get
\begin{align}\label{20200714_00:22}
 \int_{D} & \big| \big(\wM_R f_j\big)' \big| \,r^{d-1}\,\d r = \int_{D \cap D_j}   \big(\wM_R f_j\big)'  \,r^{d-1}\d r  -  \int_{D \cap C_j}   f_j' \,\,r^{d-1}\d r \nonumber\\
& \leq \int_{D \cap D_j}  f_j'\, \,r^{d-1}\,\d r + \int_{D} \big(\wM_R f\big)'  \,r^{d-1}\,\d r - \int_{D}  f'\, \,r^{d-1}\,\d r  -  \int_{D \cap C_j}   f_j' \,\,r^{d-1}\d r + o(1)  \nonumber\\
& = \int_{D} \big(\wM_R f\big)'  \,r^{d-1}\,\d r + \int_{D \cap D_j}  \big(f_j' - f'\big) \,r^{d-1}\,\d r - \int_{D \cap C_j} \big(f_j' + f'\big) \,r^{d-1}\,\d r + o(1)\\
& = \int_{D} \big(\wM_R f\big)'  \,r^{d-1}\,\d r + o(1). \nonumber
\end{align}
Note in the last passage the use of $\chi_{D \cap C_j}(r) \to 0$ and dominated convergence as in \eqref{20200713_23:28}. It follows from \eqref{20200714_00:22} that, along our subsequence of $j$’s, \begin{equation}\label{20200713_00:01}
\limsup_{j \to \infty} \int_{D}  \big| \big(\wM_R f_j\big)' \big| \,r^{d-1}\,\d r \leq \int_{D} \big|\big(\wM_R f\big)' \big| \,r^{d-1}\,\d r. 
\end{equation}
From \eqref{20200713_00:00} and \eqref{20200713_00:01} we arrive at \eqref{20200713_23:26}, and hence at \eqref{20200713_23:25}. That is,
$$(II)_j + (IV)_j = o(1).$$
Then $(IV)_j = o(1)$, and from \eqref{20200713_23:26} and \eqref{20200713_23:22} we find that $(III)_j = o(1)$ along this subsequence. This completes the proof. 

\section{Sunrise strategy reviewed: the core abstract elements} \label{Sec_strategy}A posteriori, let us take a moment to reflect on some of the main ingredients of our sunrise strategy in general terms. It should be clear by now that it is a one-dimensional mechanism, but part of its power relies on the fact that it can be applied to multidimensional maximal operators, when these act of subspaces of $W^{1,1}$ that can be identified with one-dimensional spaces.

\smallskip

Assume that we are working on a space $W^{1,1}(I, \d \mu)$, where $I \subset \R$ is an open interval or $I = \mathbb{S}^1$, and $\mu$ is a non-negative measure on $I$ such that $\mu$ and the Lebesgue measure (or arclength measure in the case of $\mathbb{S}^1$) are mutually absolutely continuous. It will be also convenient to assume that the Radon-Nikodym derivative $\frac{\d \mu}{\d x}$ is an absolutely continuous function on $I$. The cases we have in mind are: $(I, \d\mu) = (\R, \d x); \big((0,\infty), r^{d-1}\d r \big)$ for $d \geq 2$; $\big(\mathbb{S}^{1}, \d \theta \big)$; and $\big((0,\pi), (\sin \theta)^{d-1}\d \theta \big)$ for $d \geq 2$. The second option, as we have seen, appears associated to the subspace $W^{1,1}_{\rm rad}(\R^d)$ while the fourth option is associated to the subspace $W^{1,1}_{\rm pol}(\mathbb{S}^d)$. 

\smallskip

For $f\in W^{1,1}(I, \d \mu)$, that we assume non-negative and absolutely continuous in compact subsets of $I$, we let $\frak{M}$ be a maximal operator acting on $f$ such that $\frak{M}f$ is a continuous function defined on $I$. We make the additional assumption that $\frak{M}f$ is weakly differentiable and verifies the a priori bound
\begin{equation}\label{20200720_09:53}
\big\| (\frak{M}f)'\big\|_{L^1(I, \d \mu)} \lesssim_{I, \mu} \|f\|_{W^{1,1}(I, \d \mu)}.
\end{equation}
In particular, by \eqref{20200720_09:53}, $\frak{M}f$ is also absolutely continuous in compact subsets of $I$, and hence differentiable a.e. in $I$. 

\smallskip

The sunrise strategy aims to establish the continuity of the map $f \mapsto (\frak{M}f)'$, from $W^{1,1}(I, \d \mu)$ to $L^1(I, \d \mu)$. Assume that $f_j \to f$ in $W^{1,1}(I, \d \mu)$ as $j \to \infty$ (all $f_j$'s non-negative and absolutely continuous in compact subsets of $I$). As we have seen in the proof of Theorem \ref{Thm1}, the following five properties are the core elements that make the method work:
\begin{enumerate}
\item[(P1)] {\it Absence of local maxima in the disconnecting set}: $\frak{M}f$ does not have strict local maxima in the set $\{\frak{M}f > f\}$ (analogue of Proposition \ref{No_local_maxima}). 

\smallskip

\item[(P2)] {\it Convergence properties}: we have $f_j \to f$ and $\frak{M}f_j \to \frak{M}f$ pointwise in $I$ (uniformly, away from the potential singularities) and $(\frak{M}f_j)' \to (\frak{M}f)'$ pointwise a.e. in $\{\frak{M}f > f\}$ (analogue of Proposition \ref{Prop_pointwise_conv} (i) and (iii)).

\smallskip

\item[(P3)] {\it Flatness in the connecting set}: we have $f' = 0$ for a.e. point in the set $\{\frak{M}f = f\}$. This is necessary for the lateral sunrise operators to have the desired monotonicity properties of Proposition \ref{Monotonicity_lem}.

\smallskip

\item[(P4)] {\it Singularity control}: uniform control of $(\frak{M}f_j)'$ near the potential singularities (analogue of Proposition \ref{Prop_control_origin}).

\smallskip

\item[(P5)] {\it Smallness of the remainder}: Control of the remainder terms coming from the integration by parts in the final part of the proof (analogue of \eqref{20200714_00:17} - \eqref{20200714_10:04}).
 
\end{enumerate}
If these five core abstract elements are in place, the proof of Theorem \ref{Thm1} can be adapted to this situation. Note that Lemma \ref{reduction to lateral} is already in place to absorb the general setup, and our sunrise construction of the lateral operators in Section \ref{Sunrise_sec} can be performed with respect to any open interval $(\rho_1, \rho_2)$ whose closure is contained in $I \subset \R$ (this includes the whole $\R$ itself if $I = \R$), and with respect to the whole $I$ in the case $I = \mathbb{S}^{1}$.

\section{Further applications}\label{Further_app}
In this section we briefly discuss how our sunrise strategy can be applied to establish the endpoint Sobolev continuity of the other maximal operators discussed in \S \ref{Sec_Further_app_Intro}. For simplicity, the presentation here will be kept on a broad level, and we shall only indicate the major steps or changes required for each adaptation in order to verify properties (P1) - (P5) above. We omit some of the routine details.

\subsection{Proof of Theorem \ref{Thm2}} \label{Section_polar} We start by recalling that the space $W^{1,1}_{\rm pol}(\mathbb{S}^d)$ can be naturally associated to $W^{1,1}\big((0,\pi), (\sin \theta)^{d-1}\d \theta\big)$, where $\theta = \theta(\xi) = d({\bf e}, \xi)$ is the polar angle; see \cite[Lemma 13]{CGR}. For $d \geq 2$, we shall refer to $f(\xi)$ when viewing $f \in W^{1,1}_{\rm pol}(\mathbb{S}^d)$ on $\mathbb{S}^{d}$ and to $f(\theta)$ when viewing it on $(0,\pi)$. In this sense we may write
$$\|\nabla f\|_{L^1(\mathbb{S}^d)} = \omega_{d-1} \int_0^\pi |f'(\theta)| \,(\sin \theta)^{d-1}\,\d \theta.$$
Observe that inequality \eqref{20200720_12:56} accounts for \eqref{20200720_09:53} above. Properties (P1) and (P3) can be proved exactly as in \S \ref{Sec2.3_CD}. 

\smallskip

In order to verify the remaining properties, let us first consider the case $d \geq 2$. Let $g \in W^{1,1}_{\rm pol}(\mathbb{S}^d) \simeq W^{1,1}\big((0,\pi), (\sin \theta)^{d-1}\d \theta \big)$ be a given non-negative function, absolutely continuous in compact subsets of $(0,\pi)$. We start with a suitable replacement for \eqref{20200706_10:01} since we do not have the ``vanishing at infinity" situation anymore. For $0 < \theta \leq \pi/4$ we have 
\begin{align*}
 \int_0^{\frac{\pi}{4}} \widetilde{\mathcal{M}}g(\theta)\,& (\sin \theta)^{d-2} \,\cos \theta \, \d \theta   =  \int_0^{\frac{\pi}{4}} \left( \int_{\theta}^{\frac{\pi}{4}} - \big(\widetilde{\mathcal{M}}g\big)'\!(t)\, \d t + \widetilde{\mathcal{M}}g(\tfrac{\pi}{4})\right) \, (\sin \theta)^{d-2}\,\cos \theta \, \d \theta  \\
& \lesssim_d  \int_0^{\frac{\pi}{4}} \left( \int_{\theta}^{\frac{\pi}{4}} \big| \big(\widetilde{\mathcal{M}}g\big)'(t)\big|\,\d t\right) \, (\sin \theta)^{d-2}\,\cos \theta \, \d \theta   + \widetilde{\mathcal{M}}g(\tfrac{\pi}{4})\\
& =  \int_0^{\frac{\pi}{4}}  \int_{0}^{t} (\sin \theta)^{d-2}\,\cos \theta \,\,  \big| \big(\widetilde{\mathcal{M}}g\big)'(t)\big|\,\d \theta\, \d t      + \widetilde{\mathcal{M}}g(\tfrac{\pi}{4})\\
& \simeq_d \int_0^{\frac{\pi}{4}}  \big| \big(\widetilde{\mathcal{M}}g\big)'(t)\big|\, (\sin t)^{d-1} \,\d t   + \widetilde{\mathcal{M}}g(\tfrac{\pi}{4}) \\
& < \infty.
\end{align*}
Note the use of \eqref{20200720_12:56} in the last line above. An analogous computation holds in the interval $(\tfrac{3\pi}{4}, \pi)$, and also if $\widetilde{\mathcal{M}}g(\theta)$ is replaced by $g(\theta)$. If follows that the functions $\theta \mapsto g(\theta)(\sin \theta)^{d-1}$ and $\theta \mapsto \widetilde{\mathcal{M}}g(\theta)(\sin \theta)^{d-1}$ have integrable derivatives in $(0,\pi)$ and hence, by the fundamental theorem of calculus, the limits of these functions as $\theta \to 0^+$ or $\theta \to \pi^-$ must exist. If any of these limits were not zero, we would have a contradiction to the fact that $g$ and $\widetilde{\mathcal{M}}g$ belong to $L^{d/(d-1)}(\mathbb{S}^{d})$ (the former by Sobolev embedding, and the latter by the boundedness of $\widetilde{\mathcal{M}}$ in $L^{d/(d-1)}(\mathbb{S}^{d})$). Therefore
\begin{equation}\label{20200720_14:18}
\lim_{\theta \to 0^+} \!g(\theta)(\sin \theta)^{d-1} = \!\lim_{\theta \to \pi^-} \!g(\theta)(\sin \theta)^{d-1} = \!\lim_{\theta \to 0^+} \!\!\widetilde{\mathcal{M}}g(\theta)(\sin \theta)^{d-1} =\! \lim_{\theta \to \pi^-} \!\!\widetilde{\mathcal{M}}g(\theta)(\sin \theta)^{d-1} = 0.
\end{equation}

\smallskip

Given $\lambda >0$ recall now the weak-type estimate
\begin{equation}\label{20200720_13:26}
\sigma\{ \xi \in \mathbb{S}^{d} \ : \ \widetilde{\mathcal{M}}g(\xi) \geq \lambda\} \lesssim_d \frac{\|g\|_{L^1(\mathbb{S}^{d})} }{\lambda}.
\end{equation}
 Fix an interval $J_{\eta} := [\eta , \pi - \eta] \subset (0,\pi)$, say with $\eta < \tfrac{\pi}{4}$. Let $\theta_{\eta} \in J_{\eta}$ be such that $\widetilde{\mathcal{M}}g(\theta_{\eta}) = \min_{\theta \in J_{\eta}}\widetilde{\mathcal{M}}g(\theta)$. Then, taking $\lambda =  \widetilde{\mathcal{M}}g(\theta_{\eta})$ in \eqref{20200720_13:26}, we find
 \begin{equation*}
 \widetilde{\mathcal{M}}g(\theta_{\eta})  \lesssim_d \|g\|_{L^1(\mathbb{S}^{d})}.
 \end{equation*}
Hence, for any $\theta \in J_{\eta}$, we have
\begin{align}\label{20200720_14:05}
 \widetilde{\mathcal{M}}g(\theta) &= \int_{\theta_{\eta}}^{\theta}  \big(\widetilde{\mathcal{M}}g\big)'(t)\,\d t +  \widetilde{\mathcal{M}}g(\theta_{\eta})\nonumber\\
 & \lesssim_{\eta,d} \int_{0}^{\pi}  \big|\big(\widetilde{\mathcal{M}}g\big)'(t)\big|\,(\sin t)^{d-1}\d t +  \widetilde{\mathcal{M}}g(\theta_{\eta})\\
 & \lesssim_{\eta,d} \|\nabla g\|_{L^1(\mathbb{S}^{d})} + \|g\|_{L^1(\mathbb{S}^{d})}.\nonumber
\end{align}
Of course, estimates \eqref{20200720_13:26} and \eqref{20200720_14:05} also hold with $g$ replacing  $\widetilde{\mathcal{M}}g$. Then, if $f_j \to f$ in $W^{1,1}(\mathbb{S}^{d})$, an application of \eqref{20200720_14:05} with $g = f_j - f$ yields (note the sublinearity of  $\widetilde{\mathcal{M}}$) that $f_j \to f$ and $\widetilde{\mathcal{M}}f_j \to \widetilde{\mathcal{M}}f$ uniformly in the interval $J_{\eta} := [\eta , \pi - \eta]$. This is the analogue of Proposition \ref{Prop_pointwise_conv} (i). Parts (ii) and (iii) of Proposition \ref{Prop_pointwise_conv} can be proved in the same way as we did in \S \ref{Point_conv_sec} using \cite[Lemma 5]{CGR}, which is the spherical analogue of Proposition \ref{Prop_derivative_inside}. This builds up to property (P2). 

\smallskip

The analogue of Proposition \ref{Prop_control_origin}, the uniform control of $\nabla \widetilde{\mathcal{M}}f_j$ near the potential singularities (in this case, the poles ${\bf e}$ and $-{\bf e}$), can be proved in the exact same way using \eqref{20200720_14:18} and the pointwise convergence. This is property (P4). Then we proceed with the sunrise construction with respect to an open interval $(\rho, \pi - \rho)$, with $\rho$ small, and adapt the scheme of proof in Section \ref{Sect_Proof}. Note the presence of potentially two remainder terms in \eqref{20200713_12:59} coming from the integration by parts, and the proof of \eqref{20200714_10:04} will follow from directly from dominated convergence and the fact that all quantities involved are uniformly bounded in the considered interval by another application of \eqref{20200720_14:05}. This is property (P5), which completes the skeleton of the proof. We omit the remaining details of the adaptation.

\smallskip

The case $d=1$ is in fact simpler. Here our functions $f_j$ and $f$ will be absolutely continuous in the whole $\mathbb{S}^1$, and so will $\widetilde{\mathcal{M}}f_j$ and $\widetilde{\mathcal{M}}f$. Proceeding as in \eqref{20200720_13:26} and \eqref{20200720_14:05} we deduce the pointwise convergence, which is now uniform in $\mathbb{S}^1$. The analogues of Proposition \ref{Prop_pointwise_conv} (ii) and (iii) also hold. There is no need for Proposition \ref{Prop_control_origin} (property (P4)) since we do not have any singularities. We can carry out the sunrise construction with respect to the whole space $\mathbb{S}^1$ (here we must choose an orientation a priori, say clockwise, to read the decomposition \eqref{20200720_14:54}; note that the set $\widetilde{\mathcal{M}}f = f$ is always non-empty) and proceed smoothly as in Section \ref{Sect_Proof}.

\subsection{Proof of Theorem \ref{Thm_non_tang_HL}} \label{Non_tang_HL}

\subsubsection{The $\alpha =\frac{1}{3}$ threshold: a geometric argument}  \label{Threshold_sec}If $d \geq 2$ and $f \in W^{1,1}_{\rm rad}(\R^d)$ we have seen in \S \ref{Sec_basic_reg} and \S \ref{Split_Sec} that we may assume $f$ is continuous in $\R^d \setminus \{0\}$ (and non-negative for our purposes). In this case, one can verify that $M^{\alpha}f$ is also continuous in $\R^d \setminus \{0\}$, and we may also consider a degenerate cube of side zero, that is, just the point $x$ itself, in our definition of $M^{\alpha}$. As in \S \ref{Sec2.3_CD} we may define the $d$-dimensional disconnecting set 
\begin{equation*}
\mathcal{D}^{\alpha}(f) = \{x \in \R^d\setminus \{0\}\ : \ M^{\alpha} f(x) > f(x)\},
\end{equation*}
and its corresponding one-dimensional radial version
\begin{equation*}
D^{\alpha}(f) = \{ |x|\ : \ x \in \mathcal{D}^{\alpha}(f)\}.
\end{equation*}
These are open sets in $\R^d \setminus \{0\}$ and $(0,\infty)$, respectively. We define the connecting sets $\mathcal{C}^{\alpha}(f):= \big(\R^d\setminus \{0\}\big) \setminus \mathcal{D}^{\alpha}(f) $ and $C^{\alpha}(f):= (0,\infty) \setminus D^{\alpha}(f)$. In dimension $d=1$ we define the sets $D^{\alpha}(f)$ and its complement $C^{\alpha}(f)$ over the whole $\R$, for $f \in W^{1,1}(\R)$. With start by proving the analogue of Proposition \ref{No_local_maxima} in this case, a result that involves some insightful geometric considerations coming from the fact that $\alpha \geq \frac{1}{3}$.
\begin{proposition} \label{Prop_new_arg_cubes}Let $d\geq 2$ and $f \in W^{1,1}_{\rm rad}(\R^d)$. The function $M^{\alpha} f(r)$ does not have a strict local maximum in $D^{\alpha}(f)$.
\end{proposition}
\begin{proof}
Assume there is a point $r_0 \in D^{\alpha}(f)$ for which there exist $s_0$ and $t_0$ with $s_0 < r_0 < t_0$, $[s_0,t_0] \subset D^{\alpha}(f)$, such that $\wM f(r) \leq \wM f(r_0)$ for all $r \in [s_0,t_0]$ and $\wM f(s_0), \wM f(t_0) < \wM f(r_0)$. Let $x_0 \in \R^d$ be such that $|x_0| = r_0$. Let $Q_0$ be a cube such that $x_0 \in \alpha Q_0$ and 
\begin{equation*}
M^{\alpha}f(x_0) =  \intav{Q_0} f(y)\,\d y.
\end{equation*}
Observe that $Q_0$ has a positive side since $x_0 \in D^{\alpha}(f)$. Note that for any $x \in \alpha Q_0$ we have $M^{\alpha}f(x) \geq M^{\alpha}f(x_0)$, and hence $|x| \in [s_0,t_0]$ and $M^{\alpha}f(x) = M^{\alpha}f(x_0)$. This is due to the fact that the set $\{ |x| \, : \, x \in  \alpha Q_0\}$ contains $|x_0| = r_0$ and is connected. In particular, this implies that $f$ is not constant in $Q_0$, since this would contradict the fact that $M^{\alpha}f(x) >f(x)$ when $x$ is the center of $Q_0$.

\smallskip

Throughout the rest of the proof we only consider cubes with sides parallel to those of $Q_0$ (in fact, only dyadic cubes starting from $Q_0$). Let $\mathcal{A}_0 = \{Q_0\}$ and proceed inductively by defining $\mathcal{A}_k$ as the family obtained by partitioning each cube in $\mathcal{A}_{k-1}$ into $2^{d}$ dyadic cubes. Then $\mathcal{A}_k$ has $2^{dk}$ cubes of side $2^{-k}$ times the original side of $Q_0$. Since $f$ is continuous in $\R^d \setminus\{0\}$ and not constant in $Q_0$, there exists $k \geq 1$ such that the family $\mathcal{A}_k$ has a cube $Q_k$ over which we have 
\begin{equation}\label{20200722_11:19}
\intav{Q_k} f(y)\,\d y > \intav{Q_0} f(y)\,\d y  = M^{\alpha}f(x_0).
\end{equation}
Choose such $k$ minimal. We consider the genealogical sequence
$$Q_k \subset Q_{k-1} \subset \ldots Q_1 \subset Q_0,$$ 
where $Q_i \in \mathcal{A}_i$, and $Q_i$ is the parent of $Q_{i+1}$ for $i = 0,1,\ldots, k-1$. From the minimality of $k$, note that for $i =0,1, \ldots k-1$ we have
\begin{equation}\label{20200722_11:38}
 \intav{Q_i} f(y)\,\d y = \intav{Q_0} f(y)\,\d y  = M^{\alpha}f(x_0).
\end{equation}
Observe that we could not have a strictly smaller average in \eqref{20200722_11:38}, otherwise another average in the same family would be strictly larger, contradicting the minimality of $k$. 

\smallskip

If $\alpha \geq \frac{1}{3}$ we have the following relevant geometric property (recall our cubes are closed):
$$\alpha Q_i \cap \alpha Q_{i+1} \neq \emptyset$$
for any $i = 0,1,\ldots, k-1$. This means that the set $\mathcal{Y} = \cup_{i=0}^{k} \alpha Q_i$ is connected in $\R^d$ and hence its one-dimensional version, excluding the origin, $Y = \{|x|\, : \, x \in \mathcal{Y} \setminus\{0\}\}$ is also connected in $(0,\infty)$. If $x \in \mathcal{Y} \setminus\{0\}$ is such that $x \in \alpha Q_i$ for some $i =0,1, \ldots k-1$, by \eqref{20200722_11:38} we have
\begin{equation*}
 M^{\alpha} f(x) \geq \intav{Q_i} f(y)\,\d y = \intav{Q_0} f(y)\,\d y  = M^{\alpha}f(x_0).
\end{equation*}
If $x \in \alpha Q_k$, by \eqref{20200722_11:19} we have
$$M^{\alpha} f(x) \geq \intav{Q_k} f(y)\,\d y > \intav{Q_0} f(y)\,\d y  = M^{\alpha}f(x_0).$$
Hence $Y$ is a connected set in $(0,\infty)$ (i.e. an interval) such that: (i) it contains $r_0 = |x_0|$; (ii) $M^{\alpha} f(r) \geq M^{\alpha} f(r_0)$ for every $r \in Y$; (iii) there is a point $r_k = |x|$ (with $x \in \alpha Q_k$) in $Y$ such that $M^{\alpha} f(r_k) > M^{\alpha} f(r_0)$. This contradicts the fact that $r_0$ was a strict local maximum.
\end{proof}
\noindent{\sc Remark:} The proof of Proposition \ref{Prop_new_arg_cubes} can be modified to the case of dimension $d=1$ and a function $f:\R \to \R$ that is continuous and of bounded variation. In this case we also have $M^{\alpha} f$ continuous and a strict local maximum in the disconnecting set would have $M^{\alpha} f$ realized in a bounded and non-denegerate interval. This provides an alternative approach to \cite{Ra} in order to prove \eqref{20200720_13:11}.

\begin{figure}
  \centering
  \includegraphics[height=7.5cm]{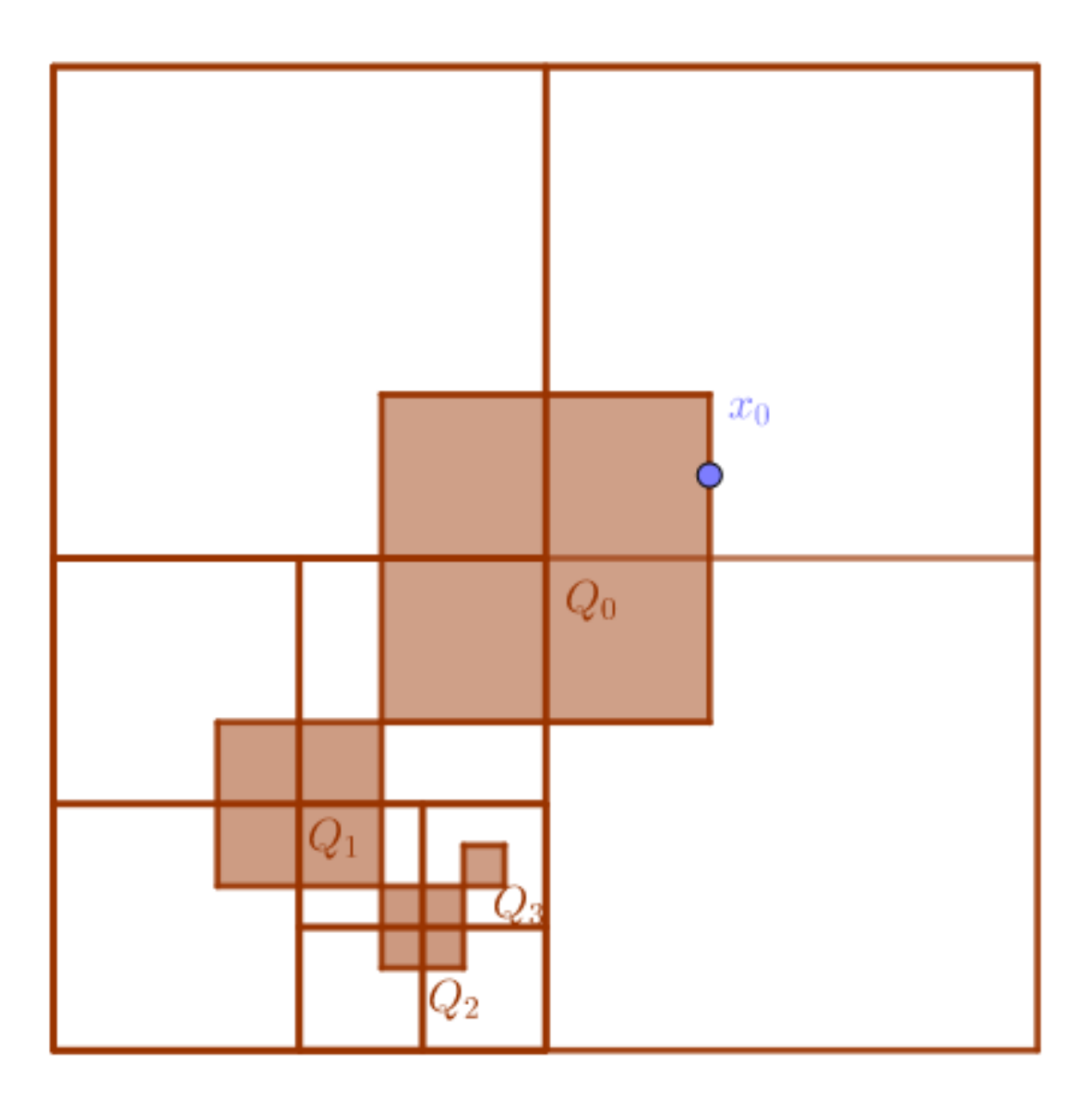} 
    \caption{Illustration of the construction in the case $\alpha =\frac{1}{3}$. The dyadic cubes $Q_0, Q_1, Q_2, Q_3$ are in white, and the colored cubes represent $\alpha Q_i$ ($i =0,1,2,3$).}
\label{fig:Cubos}
\end{figure}

\smallskip

We now proceed to the proof of Theorem \ref{Thm_non_tang_HL}.

\subsubsection{Proof of Theorem \ref{Thm_non_tang_HL}: boundedness} \label{Bound_NTHL}We first briefly consider the boundedness claim in part (ii). Here $d \geq 2$. Observe first that 
\begin{equation}\label{20200722_13:08}
M^{\alpha}f(x) \lesssim_{d, \alpha} \wM f(x).
\end{equation}
One now proceeds via the following steps:

\smallskip

\noindent {\it Step} 1. Show that $M^{\alpha}f$ is locally Lipschitz in the disconnecting set $\mathcal{D}^{\alpha}(f)$. For this, note that every $x \in \mathcal{D}^{\alpha}(f)$ has a neighborhood $x \in U_x \subset \mathcal{D}^{\alpha}(f)$ in which the cubes that realize the maximal function for any $y \in U_x$ are of size bounded by below. Take two points $y,z \in U_x$ and compare their maximal functions by using translated cubes and the fact that the difference quotients are uniformly bounded in $L^1$ by a multiple of the $L^1$-norm of the gradient of $f$. Hence $M^{\alpha}f$ is differentiable a.e. in $\mathcal{D}^{\alpha}(f)$.

\smallskip

\noindent {\it Step} 2. Follow line-by-line the mechanism of proof of the first two authors in \cite[Theorem 1, \S 2.1]{CGR} to prove that 
\begin{equation}\label{20200726_08:54}
\int_{D^{\alpha}(f)} \big| \big(M^{\alpha}f\big)'(r)\big| \, r^{d-1}\,\d r \lesssim_{d, \alpha} \int_{0}^{\infty} |f'(r)| \, r^{d-1}\,\d r.
\end{equation}
This scheme, which in \cite{CGR} is used for maximal functions of convolution type, only requires the control \eqref{20200722_13:08}, the bound \eqref{20200702_12:24}, and the absence of local maxima in the disconnecting set given by Proposition \ref{Prop_new_arg_cubes}. 

\smallskip

\noindent {\it Step} 3. Follow line-by-line the argument in \cite[\S 5.4]{CS} to show that $M^{\alpha}f(r)$ is weakly differentiable in $(0,\infty)$ with weak derivative given by $\chi_{C^{\alpha}(f)}f' + \chi_{D^{\alpha}(f)}(M^{\alpha}f)'$. Conclude that $M^{\alpha}f(x)$ is weakly differentiable in $\R^d$ by the discussion in \S \ref{Sec_basic_reg} (see \cite[Lemma 4]{CGR}) and that the desired bound
\begin{equation*}
\|\nabla M^{\alpha}f\|_{L^1(\R^d)} \lesssim_{d, \alpha} \, \|\nabla f\|_{L^1(\R^d)}
\end{equation*}
follows from \eqref{20200726_08:54}.

\subsubsection{Proof of Theorem \ref{Thm_non_tang_HL}: continuity}  Let us look at properties (P1) - (P5) described in Section \ref{Sec_strategy}. We have already established (P1). Let us move to property (P2). The uniform pointwise convergence $M^{\alpha}f_j(r) \to M^{\alpha}f(r)$ follows from the sublinearity of $M^{\alpha}$, together with \eqref{20200722_13:08} and \eqref{20200724_11:39}. For the convergence of the derivatives a.e. in the disconnecting set $D^{\alpha}(f)$ one may start establishing an analogue of Proposition \ref{Prop_derivative_inside} to move the derivative inside an average over a ``good" cube; this follows with the same proof, that only uses translations in $\R^d$. One also needs the analogue of Proposition \ref{Prop_pointwise_conv} (ii) on accumulating sequences of ``good cubes". Here the proof is also the same, and one may think of parametrizing the cubes by its center, its side and its orientation (say, with a set of $d$ orthogonal vectors in $\mathbb{S}^{d-1}$). This leads to the desired analogue of Proposition \ref{Prop_pointwise_conv} (iii).

\smallskip

Establishing (P3) requires a brief computation and we do it for $d\geq 2$  in the next proposition (the case $d\geq 1$ and $I = \R$ being easier and following via the same reasoning). 
\begin{proposition}\label{Prop_discon_}
Let $\alpha >0$, $d \geq 2$ and $f \in W^{1,1}_{\rm rad}(\R^d)$. Let $r_0>0$ be a point of differentiability of $f(r)$ such that $f'(r_0) \neq 0$. Then $r_0 \in D^{\alpha}(f)$.
\end{proposition}
\begin{proof}
Assume first that $f'(r_0) = c >0$. Take a point $x_0 = (r_0, 0, \ldots, 0) \in \R^d$. For $h >0$, we consider a cube $Q_h$ with sides parallel to the usual axes, with side length $2h$, and center $z_0 = (r_0 + \alpha h, 0, \ldots, 0)$. Note that $x_0$ belongs to the boundary of $\alpha Q_h$. The idea is to have $Q_h$ ``to the right of $x_0$" as much as possible. If $\alpha \geq1$, we see that this cube is completely to the right of $x_0$ and for $h$ small we can easily infer that $\intav{Q_h} f > f(x_0)$. If $\alpha <1$, part of this cube will be ``to the left of $x_0$" and we must be a bit more careful. Fix $\varepsilon >0$ small (say, with $ \varepsilon < \min\{c,1\} $ to begin with). Then 
\begin{equation}\label{20200730_22:15}
f(r_0 + s) \geq f(r_0) + (c - \varepsilon)s  \ \ {\rm and} \ \ f(r_0 - s) \geq f(r_0) - (c + \varepsilon)s
\end{equation}
for $|s| \leq s_0(\varepsilon)$. Assume $h$ is sufficiently small so that $\big||y| - r_0\big| \leq s_0(\varepsilon)$ for all $y \in Q_h$. Then, letting $y = (y_1, y_2, \ldots, y_d)$ be our variable in $\R^d$, using \eqref{20200730_22:15} and the basic fact that $|y| \geq y_1$ we get 
\begin{align}\label{20200724_11:16}
\begin{split}
\intav{Q_h} f(y)\, \d y  & -  f(x_0)  = \frac{1}{|Q_h|} \left( \int_{y \in Q_h \, : \, |y| \geq r_0} f(y)\, \d y  + \int_{y \in Q_h\,:\,|y| < r_0} f(y)\, \d y\right) -  f(x_0) \\
& \geq \frac{1}{|Q_h|} \left( \int_{y \in Q_h\, :\, |y| \geq r_0} (c - \varepsilon) (|y| - r_0)\, \d y  + \int_{y \in Q_h\, :\,|y| < r_0} (c + \varepsilon)(|y| - r_0)\, \d y\right)\\
&  \geq \frac{1}{|Q_h|} \left( \int_{y \in Q_h\, :\, y_1 \geq r_0} (c - \varepsilon) (y_1 - r_0)\, \d y  + \int_{y \in Q_h\, :\,y_1 < r_0} (c + \varepsilon)(y_1 - r_0)\, \d y\right)\\
& = \frac{1}{2h} \left(  \int_{r_0}^{r_0 + \alpha h + h} (c - \varepsilon) (y_1 - r_0)\, \d y_1 + \int_{r_0 + \alpha h - h}^{r_0}(c + \varepsilon) (y_1 - r_0)\, \d y_1\right)\\
& = \frac{1}{2h}  \left( (c - \varepsilon) \frac{(\alpha h  +h)^2}{2} - (c +\varepsilon) \frac{(\alpha h -h)^2}{2}\right)\\
& = \frac{h( 2 \alpha c -  \varepsilon \alpha^2 - \varepsilon)}{2}.
\end{split}
\end{align}
The latter is strictly positive as long as we choose $\varepsilon < 2 \alpha c/ (\alpha^2 +1)$, which is clearly possible if $\alpha >0$. A similar argument shows that if $f'(r_0) = c <0$ then $r_0 \in D^{\alpha}(f)$. Here we choose $z_0 = (r_0 - \alpha h, 0, \ldots, 0)$, and choose $h$ small so that if $y \in Q_h$ then $y_1 \geq r_0 - \alpha h - h \geq  \frac{r_0}{2}$ and 
\begin{equation}\label{20200730_22:40}
|y| \leq \sqrt{y_1^2 + (d-1)h^2} \leq \left(y_1 + \frac{(d-1)h^2}{2y_1}\right) \leq y_1 + \varepsilon h.
\end{equation}
We use \eqref{20200730_22:40} in the analogue of passage \eqref{20200724_11:16}.
\end{proof}

Property (P4) is not needed in the case $d=1$, whereas in the case $d\geq 2$ we can prove it following the same outline of Proposition \ref{Prop_control_origin}, with minor adjustments to allow for a dependence on $\alpha$. The we perform the surnrise construction, in the case $d=1$ with respect to the whole $\R$, and in the case $d\geq2$ as we already did, in an interval $(\rho, \infty)$. The proof in Section \ref{Sect_Proof} goes through identically, as \eqref{20200722_13:08} can be used to prove the analogue of \eqref{20200714_10:04} (property (P5)).

\subsection{Proof of Theorem \ref{Thm_non_tang_heat_flow}} \label{Proof_heat_flow} We start by observing that, for any $\alpha \geq 0$, we have the pointwise bound (see \cite[Chapter II, Eq. (3.18)]{SteinWeiss})
\begin{equation}\label{20200728_10:00}
M^{\alpha}_{\varphi}f(x) \lesssim_{d, \alpha} \wM f(x).
\end{equation}
In the rest of the proof we focus in the case $d\geq 2$. The case $d=1$ is simpler and requires only minor modifications. We start with the usual setup, in which our $f \in W^{1,1}_{\rm rad}(\R^d)$ is non-negative and continuous in $\R^d \setminus \{0\}$, and one can verify that $M^{\alpha}_{\varphi}f$ is also radial and continuous in $\R^d \setminus \{0\}$.

\subsubsection{Absence of local maxima} Define the disconnecting sets $\mathcal{D}^{\alpha}_{\varphi}(f)$ (in $\R^d \setminus \{0\})$ and $D^{\alpha}_{\varphi}(f)$ (in $(0,\infty)$), and the connecting sets $\mathcal{C}^{\alpha}_{\varphi}(f)$ and $C^{\alpha}_{\varphi}(f)$ as we already did in \S \ref{Sec2.3_CD} or \S \ref{Threshold_sec}. We first establish property (P1), the analogue of Proposition \ref{No_local_maxima}. 
\begin{proposition} \label{Prop_no_max_heat_flow}Let $\alpha > 0$, $d\geq 2$ and $f \in W^{1,1}_{\rm rad}(\R^d)$. The function $M^{\alpha}_{\varphi} f(r)$ does not have a strict local maximum in $D^{\alpha}_{\varphi}(f)$.
\end{proposition}
\begin{proof}
Assume there is a point $r_0 \in D_{\varphi}^{\alpha}(f)$ for which there exist $s_0$ and $t_0$ with $s_0 < r_0 < t_0$, $[s_0,t_0] \subset D_{\varphi}^{\alpha}(f)$, such that $M^{\alpha}_{\varphi} f(r) \leq M^{\alpha}_{\varphi} f(r_0)$ for all $r \in [s_0,t_0]$ and $M^{\alpha}_{\varphi} f(s_0), M^{\alpha}_{\varphi} f(t_0) < M^{\alpha}_{\varphi} f(r_0)$. Let $x_0 \in \R^d$ be such that $|x_0| = r_0$. Assume that $M^{\alpha}_{\varphi} f(x_0) = f*{\varphi}_t(z_0)$ with $|z_0 - x_0| \leq \alpha \sqrt{t}$. For any $y \in \overline{B_{\alpha \sqrt{t}}(z_0)}$ note that the pair $(z_0, t)$ is an admissible choice for the maximal function $M^{\alpha}_{\varphi} f$ at $y$, hence $M^{\alpha}_{\varphi} f(y) \geq M^{\alpha}_{\varphi} f(x_0)$. Since $x_0$ is a strict local maximum, in our setup we must then have $\big\{|y|\,:\, y \in \overline{B_{\alpha \sqrt{t}}(z_0)}\big\} \subset [s_0,t_0]$ and $M^{\alpha}_{\varphi} f(y) = M^{\alpha}_{\varphi} f(x_0) = f*{\varphi}_t(z_0)$ for such $y$. In particular this implies that $z_0 \neq 0$ and that $M^{\alpha}_{\varphi} f(z_0) = M^{0}_{\varphi} f(z_0) = f*{\varphi}_t(z_0) > f(z_0)$. Hence $|z_0|$ is a strict local maximum of  $M^{0}_{\varphi} f(r)$ in the disconnecting set $D^{0}_{\varphi}(f)$. This contradicts \cite[Lemma 8]{CS}, i.e. the fact that $M^{0}_{\varphi} f$ is subharmonic in the disconnecting set (which is the case $\alpha = 0$ of this proposition). Note that \cite[Lemma 8]{CS} is originally stated for continuous functions $f$ but its proof only uses such continuity in a neighborhood of $z_0$ whose closure is contained in the disconnecting set $\mathcal{D}^{0}_{\varphi}(f)$ (which serves our purposes here).
\end{proof}

\subsubsection{Proof of Theorem \ref{Thm_non_tang_heat_flow}: boundedness} Once we have \eqref{20200728_10:00} and Proposition \ref{Prop_no_max_heat_flow} in our hands, the proof of the boundedness follows the exact same outline with three steps of \S \ref{Bound_NTHL} (in Step 1, one would think of the time $t$ being bounded by below).

\smallskip

Having gone through the three steps above and established the gradient bound, it will be useful to take a closer look at the second step, for it provides, as a corollary, a local estimate that will imply our desired property (P4). Let $\rho >0$ and write
$$D^{\alpha}_{\varphi}(f) \cap (0,\rho) = \bigcup_{i=1}^{\infty} (a_i, b_i).$$
For each $i$, let $\tau_i \in [a,b]$ be a point of minimum of $M^{\alpha}_{\varphi} f$ is such interval (then $M^{\alpha}_{\varphi} f$ is non-increasing in $[a_i,\tau_i]$ and non-decreasing in $[\tau_i, b_i]$). Assuming for a moment that $b_i \neq \rho$, using integration by parts we get
\begin{align}\label{20200728_10:53}
 \int_{a_i}^{b_i} & \big| \big(M^{\alpha}_{\varphi} f\big)'(r)\big| \,r^{d-1}\,\d r = - \int_{a_i}^{\tau_i}  \big(M^{\alpha}_{\varphi} f\big)'(r) \,r^{d-1} \,\d r + \int_{\tau_i}^{b_i}  \big(M^{\alpha}_{\varphi} f\big)'(r)\,r^{d-1}\,\d r  \nonumber\\
& = M^{\alpha}_{\varphi} f(a_i)\,{a_i}^{d-1} + M^{\alpha}_{\varphi} f(b_i)\,{b_i}^{d-1} - 2M^{\alpha}_{\varphi} f(\tau_i)\,\tau_{i}^{d-1}   \nonumber\\
& \ \ \ \ \ \ \ \ \ \ \ + (d-1) \int_{a_i}^{\tau_i}  M^{\alpha}_{\varphi} f(r) \,r^{d-2}\,\d r - (d-1)\int_{\tau_1}^{b_i}  M^{\alpha}_{\varphi} f(r) \,r^{d-2}\,\d r  \nonumber\\
& \lesssim_{d,\alpha} f(a_i)\,{a_i}^{d-1} + f(b_i)\,{b_i}^{d-1} - 2f(\tau_i)\,\tau_i^{d-1}\\
& \ \ \ \ \ \ \ \ \ \ \ + (d-1) \int_{a_i}^{\tau_i}  \wM f(r) \,r^{d-2}\,\d r - (d-1)\int_{\tau_i}^{b_i}  f(r) \,r^{d-2}\,\d r  \nonumber\\
& = f(a_i)\,{a_i}^{d-1} - f(\tau_i)\,\tau_i^{d-1}  + (d-1) \int_{a_i}^{\tau_i}  \wM f(r) \,r^{d-2}\,\d r  + \int_{\tau_i}^{b_i}  f'(r) \,r^{d-1}\,\d r  \nonumber\\
&\leq (d-1) \int_{a_i}^{\tau_i}  \wM f(r) \,r^{d-2}\,\d r  + \int_{a_i}^{b_i}  |f'(r)| \,r^{d-1}\,\d r. \nonumber
\end{align}
The last inequality holds since
$$f(a_i)\,{a_i}^{d-1} - f(\tau_i)\,\tau_i^{d-1} \leq - \int_{a_i}^{\tau_i} f'(r)\, r^{d-1}\,\d r \leq \int_{a_i}^{\tau_i} |f'(r)|\, r^{d-1}\,\d r.$$
From \eqref{20200728_10:00} and \eqref{20200706_10:24} note that there is no issue in \eqref{20200728_10:53} if $a_i =0$. If $b_i = \rho$, the inequality \eqref{20200728_10:53} continues to hold if we add a term $\wM(\rho)\, \rho^{d-1} - f(\rho)\, \rho^{d-1}$ on the right hand-side. If we sum over all intervals (and take also the connecting set into consideration) we arrive at 
\begin{equation}\label{20200728_11:12}
\int_0^{\rho} \big| \big(M^{\alpha}_{\varphi} f\big)'(r)\big| \,r^{d-1}\,\d r \lesssim_{d,\alpha} \int_{0}^{\rho}  |f'(r)| \,r^{d-1}\,\d r + \int_{0}^{\rho}  \wM f(r) \,r^{d-2}\,\d r + \wM(\rho)\, \rho^{d-1}.
\end{equation}
On the other hand, a similar computation to \eqref{20200706_10:01} yields
\begin{align}\label{20200728_11:13}
\begin{split}
\int_{0}^{\rho}  \wM f(r) \,r^{d-2}\,\d r &= \int_{0}^{\rho} \left( \wM f(\rho) - \int_{r}^{\rho} \big( \wM f\big)'(t)\,\d t \right) r^{d-2}\,\d r\\
& \lesssim_d \wM f(\rho) \,\rho^{d-1} + \int_0^{\rho} \big|\big( \wM f\big)'(t)\big|\,t^{d-1}\,\d t.
\end{split}
\end{align}
Combining \eqref{20200728_11:12} and \eqref{20200728_11:13} we arrive at
\begin{equation*}
\int_0^{\rho} \big| \big(M^{\alpha}_{\varphi} f\big)'(r)\big| \,r^{d-1}\,\d r \lesssim_{d,\alpha} \int_{0}^{\rho}  |f'(r)| \,r^{d-1}\,\d r + \int_0^{\rho} \big|\big( \wM f\big)'(r)\big|\,r^{d-1}\,\d r + \wM(\rho)\, \rho^{d-1}.
\end{equation*}
Observe that this estimate, combined with Proposition \ref{Prop_control_origin}, plainly yields the analogue of Proposition \ref{Prop_control_origin} for the non-tangential operators $M^{\alpha}_{\varphi}$. This is property (P4) in our to-do list (which is not needed for the case $d=1$).

\subsubsection{Proof of Theorem \ref{Thm_non_tang_heat_flow}: continuity} We have already established properties (P1) and (P4) of our sunrise strategy outlined in Section \ref{Sec_strategy}. Property (P2) follows pretty much as in Proposition \ref{Prop_pointwise_conv}, using \eqref{20200728_10:00} and the sublinearity of $M^{\alpha}_{\varphi} f$ for the convergences at the function level, and verifying that one can move the gradient inside the integral as in Proposition \ref{Prop_derivative_inside} in the disconnecting set. The sunrise construction will be identical to Section \ref{Sunrise_sec} when $d\geq 2$ (and over $I = \R$ when $d=1$) and one shall use \eqref{20200728_10:00} to prove the analogue of \eqref{20200714_10:04} (property (P5)). The proof will be complete once we establish property (P3). This is the content of our final proposition (which also holds for $d=1$ and $I=\R$ with the same reasoning).

\begin{proposition}\label{Det_Heat_Flat}
Let $\alpha > 0$, $d\geq 2$ and $f \in W^{1,1}_{\rm rad}(\R^d)$. Let $r_0>0$ be a point of differentiability of $f(r)$ such that $f'(r_0) \neq 0$. Then $r_0 \in D_{\varphi}^{\alpha}(f)$.
\end{proposition}
\begin{proof}
The proof here is similar in spirit to the proof of Proposition \ref{Prop_discon_}, but technically slightly more involved. We first consider the case $f'(r_0) = c >0$ and let $x_0 := (r_0, 0, \ldots, 0) \in \R^d$. Fix $\varepsilon >0$ small (say, with $ \varepsilon < \min\{c,1\} $ to begin with).  Then we have 
\begin{equation}\label{20200730_13:42}
f(r_0 + s) \geq f(r_0) + (c - \varepsilon)s  \ \ {\rm and} \ \ f(r_0 - s) \geq f(r_0) - (c + \varepsilon)s
\end{equation}
for $|s| \leq s_0(\varepsilon)$. 

\smallskip

For $t <1$ small we set $N := (t \varepsilon)^{-1/8}$ and consider the cube $Q_t$ of center at the origin and side $2N\sqrt{t}$ (with sides parallel to the usual axes). We let $z_0 = (r_0 + \alpha\sqrt{t}, 0, 0 , \ldots, 0)$ and we want to show that $f*\varphi_t(z_0) > f(x_0)$ when $t$ and $\varepsilon$ are small enough (note that we are trying to place the mass of the heat kernel ``to the right" of $r_0$). Since the heat kernel is radial we may write
\begin{align*}
f*\varphi_t(z_0) - f(x_0) = \int_{\R^d} \big(f(z_0 +y) - f(x_0)\big) \, \varphi_t(y)\,\d y = \int_{Q_t} + \ \int_{Q_t^c} =: (I) + (II).
\end{align*}
We first verify that the integral $(II)$ is small. By the Sobolev embedding, recall that $f \in L^{d/(d-1)}(\R^d)$. Observe also that 
\begin{align}\label{20200730_21:05}
\begin{split}
\int_{Q_t^c}  & \varphi_t(y)^d \,\d y \leq \int_{|y| \geq N\sqrt{t}}\ \varphi_t(y)^d\,\d y  = \omega_{d-1} \int_{N\sqrt{t}}^{\infty} \frac{s^{d-1}}{(4 \pi t)^{d^2/2}} \, e^{\frac{-d\, s^2}{4t}}\d s  \\
& = \frac{\omega_{d-1}}{t^{d(d-1)/2}} \int_{N}^{\infty} \frac{u^{d-1}}{(4 \pi)^{d^2/2}} \, e^{\frac{-d\, u^2}{4}}\d u \lesssim_d \frac{e^{\frac{-d N^2}{8}}}{t^{d(d-1)/2}} \lesssim_d \frac{N^{4d(-d(d-1)-1)}}{t^{d(d-1)/2}} \leq \big(\varepsilon \sqrt{t}\big)^d.
\end{split}
\end{align}
Hence, using H\"{o}lder's inequality we get 
\begin{align}\label{20200730_13:05}
\int_{Q_t^c}f(z_0 +y)  \, \varphi_t(y)\,\d y \leq \|f\|_{L^{d/(d-1)}(\R^d)} \left(\int_{Q_t^c} \varphi_t(y)^d \,\d y \right)^{1/d} \lesssim_d \|f\|_{L^{d/(d-1)}(\R^d)} \ \varepsilon \sqrt{t}.
\end{align}
Similarly, one can show that
\begin{align}\label{20200730_13:08}
\int_{Q_t^c}  f(x_0) \,  \varphi_t(y)\,\d y  \lesssim_d f(x_0) \, \varepsilon \sqrt{t}.
\end{align}
Combining \eqref{20200730_13:05} and \eqref{20200730_13:08} we arrive at 
\begin{equation}\label{20200730_21:13}
(II) = \varepsilon \sqrt{t} \, O(1),
\end{equation}
where the implicit constant depends only on $d, \|f\|_{L^{d/(d-1)}(\R^d)}$ and $f(x_0)$. 

\smallskip

We then move to the analysis of the term $(I)$. Let $Q_t = Q_t^{+} \cup Q_t^-$, where $Q_t^+ = \{y \in Q_t\, : \, |z_0 + y| \geq r_0\}$ and $Q_t^- = \{y \in Q_t\, : \, |z_0 + y| < r_0\}$. Assume $t$ is sufficiently small so that $\big||z_0 + y| - r_0\big| \leq s_0(\varepsilon)$ for all $y \in Q_t$. Then, letting $y = (y_1, y_2, \ldots, y_d)$, using \eqref{20200730_13:42} and the fact that $|z_0+y| \geq r_0 + \alpha \sqrt{t} + y_1$, we get
\begin{align}\label{20200730_14:14}
& \int_{Q_t}  \big(f(z_0 +y) - f(x_0)\big) \, \varphi_t(y)\,\d y = \int_{Q_t^+} \, + \ \int_{Q_t^-} \nonumber \\
& \geq  \int_{Q_t^+} (c - \varepsilon) \big(|z_0 +y| - r_0\big)\, \varphi_t(y)\,\d y  - \int_{Q_t^-} (c + \varepsilon) \big(-|z_0 +y| + r_0\big)\, \varphi_t(y)\,\d y \nonumber \\
& \geq \! \int_{Q_t^+} (c - \varepsilon) \big((r_0 + \alpha \sqrt{t} + y_1) - r_0\big) \varphi_t(y)\,\d y -  \int_{Q_t^-} (c + \varepsilon) \big(\!-(r_0 + \alpha \sqrt{t} + y_1) + r_0\big) \varphi_t(y)\,\d y\\
& = c \,\alpha \sqrt{t} \int_{Q_t} \varphi_t(y)\,\d y + \varepsilon \left( - \int_{Q_t^+} (\alpha \sqrt{t} + y_1) \, \varphi_t(y)\,\d y + \int_{Q_t^-} (\alpha \sqrt{t} + y_1) \, \varphi_t(y)\,\d y \right). \nonumber
\end{align}
Note that we used above the fact that $\int_{Q_t} y_1 \varphi_t(y)\,\d y = 0$, since $\varphi_t$ is even. Proceeding as in \eqref{20200730_21:05} and \eqref{20200730_13:08} we find that 
\begin{equation}\label{20200730_14:15}
1-  O_d(\varepsilon) \leq \int_{Q_t} \varphi_t(y)\,\d y \leq  1.
\end{equation}
and
\begin{align}\label{20200730_14:16}
\int_{Q_t} |y_1| \,\varphi_t(y)\,\d y & \leq  \int_{-N \sqrt{t}}^{N \sqrt{t}} \frac{|y_1|}{(4 \pi t)^{1/2}} \, e^{\frac{-y_1^2}{4t}}\,\d y_1  = \sqrt{t} \int_{-N }^{N} \frac{|u|}{(4 \pi)^{1/2}} \, e^{\frac{-u^2}{4}}\d u \lesssim \sqrt{t}.
\end{align}
Using \eqref{20200730_14:15} and \eqref{20200730_14:16} in \eqref{20200730_14:14} we arrive at 
\begin{equation}\label{20200730_21:14}
(I) = \int_{Q_t}  \big(f(z_0 +y) - f(x_0)\big) \, \varphi_t(y)\,\d y \geq  \sqrt{t} \, \Big( c \,\alpha \big(1 - O_d(\varepsilon)\big) - \varepsilon \big(\alpha + O(1)\big) \Big).
\end{equation}
Note that the work in \eqref{20200730_21:13} and \eqref{20200730_21:14} had the intention of leaving things in the same scale $\sqrt{t}$. Combining \eqref{20200730_21:13} and \eqref{20200730_21:14} we arrive at 
\begin{equation}\label{20200730_21:19}
(I) + (II) \geq \sqrt{t} \, \Big( c \,\alpha \big(1 - O_d(\varepsilon)\big) - \varepsilon \big(\alpha + O(1)\big) \Big),
\end{equation}
where the implicit constant in the $O(1)$ depends only on $d, \|f\|_{L^{d/(d-1)}(\R^d)}$ and $f(x_0)$. Since $c >0$ and $\alpha >0$, the conclusion is that for our initial choice of $\varepsilon$ sufficiently small we will have \eqref{20200730_21:19} strictly positive, as we wanted.

\smallskip

The case $f'(r_0) = c <0$ follows along the same lines. Given our initial $\varepsilon >0$, we will now choose $z_0 = (r_0 - \alpha\sqrt{t}, 0, 0 , \ldots, 0)$. We start with $t$ small so that $r_0 - \alpha\sqrt{t} - N\sqrt{t} \geq \frac{r_0}{2}$. Then we can go to $t$ even smaller such that for every  $y \in Q_t$ we have
\begin{align*}
|z_0+y| & \leq \Big((r_0 - \alpha \sqrt{t} + y_1)^2 + (d-1)N^2 t\Big)^{1/2} \leq (r_0 - \alpha \sqrt{t} + y_1) + \frac{(d-1)N^2t }{2(r_0 - \alpha \sqrt{t} + y_1)}\\
& \leq (r_0 - \alpha \sqrt{t} + y_1) + \varepsilon \sqrt{t}.
\end{align*}
We use this inequality in the analogue of \eqref{20200730_14:14}.
\end{proof}

\subsection{Concluding remarks}\label{Conc_rem} We briefly comment on the obstructions towards the endpoint $W^{1,1}$--continuity via the sunrise strategy for some maximal operators mentioned, or at least hinted at, in our text (and for which the corresponding boundedness result is already established). The non-tangential Hardy-Littlewood maximal operator $M^{\alpha}$, in the case of dimension $d=1$ and $0 < \alpha < \frac{1}{3}$, does not necessarily verify property (P1) as exemplified in \cite[Theorem 2]{Ra} (think of $f$ being two high bumps far apart). Still in dimension $d=1$, for the centered Hardy-Littlewood maximal operator, on top of obstruction (P1), property (P3) may also not be verified. The centered heat flow maximal function $M^{0}_{\varphi}$ (in dimension $d=1$ for general $f \in W^{1,1}(\R)$ and if $d\geq 2$ for $f \in W^{1,1}_{\rm rad}(\R^d)$) verifies (P1) but does not necessarily verify the flatness property (P3) (just think of $f$ being the Gaussian $\varphi_1$). 

\smallskip

Another standard maximal function of convolution type is the one associated to the Poisson kernel
\begin{equation*}
\Psi_t(x) = \frac{\Gamma \left(\frac{d+1}{2}\right)}{\pi^{(d+1)/2}}\ \frac{t}{(|x|^2 + t^2)^{(d+1)/2}}.
\end{equation*}
Similarly to \eqref{20200727_11:41}, for $\alpha \geq 0$ we may consider 
\begin{equation}\label{20200727_11:41}
M_{\Psi}^{\alpha}f(x) = \sup_{t >0 \, ; \, |y-x|\leq \alpha t} \,(|f| * \Psi_t)(y).
\end{equation}
The boundedness of the map $f \to \big(M_{\Psi}^{\alpha}f\big)'$ from $W^{1,1}(\R) \to L^1(\R)$ was established for $\alpha =0$ in \cite[Theorem 2]{CS} and for $\alpha >0$ in \cite[Theorem 4]{CFS}. When $d\geq 2$ and $\alpha =0$ the boundedness of the map $f \to \nabla M_{\Psi}^{0}f$ from $W^{1,1}_{\rm rad}(\R^d) \to L^1(\R^d)$ was established in \cite[Theorem 1]{CGR}. Following the exact same argument of our Theorem \ref{Thm_non_tang_HL} we can extend this boundedness result in dimension $d\geq 2$ for $\alpha >0$ as well (this has not been recorded in the literature before). In all of the cases above, property (P1) holds; and this is actually an important ingredient in such boundedness proofs. One may be naturally led to think that the analogue of Proposition \ref{Det_Heat_Flat}, i.e. property (P3), would be somewhat reasonable for such an operator, at least in the non-tangential case $\alpha >0$. This turns out to be false. The flatness property (P3) is not necessarily verified for any $\alpha \geq 0$. 

\smallskip

\enlargethispage{0.4\baselineskip}

In dimension $d\geq 2$, it is shown in \cite[\S 5.3]{CFS} that the function
$$f(x) = (1 + |x|^2)^{\frac{-d+1}{2}}$$
is such that $M_{\Psi}^{\alpha}f(x) = f(x)$ for $|x| \leq \frac{1}{\alpha}$. Such $f$ is not in $W^{1,1}(\R^d)$, but we could simply multiply $f$ by a smooth and radially non-increasing function $\phi$ with $\phi(x) = 1$ if $|x| \leq 1$, and $\phi(x) = 0$ if $|x| \geq 2$, that the property $M_{\Psi}^{\alpha}f(x) = f(x)$ would continue to hold in a neighborhood of the origin. In dimension $d=1$ we may consider the function 
\begin{equation*}
f(x) = \log\left( \frac{4 + x^2}{1 + x^2}\right) = 2 \int_1^2 \frac{s}{(s^2 + x^2)} \, \d s = 2\pi \int_1^2 \Psi_s(x) \, \d s
\end{equation*}
This function belongs to $W^{1,1}(\R)$. Using the semigroup property of the Poisson kernel we get
\begin{align*}
v(y,t) & :=  (f * \Psi_t)(y) = 2 \pi \int_{-\infty}^{\infty}\int_1^2 \Psi_s(y-x) \,\Psi_t (x) \, \d s \, \d x =  2 \pi \int_1^2 \int_{-\infty}^{\infty}\Psi_s(y-x) \,\Psi_t (x) \,\d x\, \d s \\
& = 2 \pi \int_1^2 \Psi_{t+s} (y)\,\d s = \log\left( \frac{(t+2)^2 + y^2}{(t+1)^2 + y^2}\right).
\end{align*}
For a fixed $x \in \R$, by the maximum principle (recall that $v$ verifies $\Delta v = 0$ in $\R \times (0,\infty)$), the supremum of $v(y,t)$ in the cone $|y-x| \leq \alpha t$ is attained at a point $y = x \pm \alpha t$. We want to show that, for $x$ in a neighborhood of the origin we have
\begin{equation*}
\log\left( \frac{4 + x^2}{1 + x^2}\right) \geq \log\left( \frac{(t+2)^2 + (x \pm \alpha t)^2}{(t+1)^2 + (x \pm \alpha t)^2}\right)
\end{equation*}
for all $t \geq 0$. After removing the $\log$ and multiplying out, this is equivalent to 
$$t\, (-2 x^2 \pm 6 x\alpha  + 3 \alpha^2 t + 3t + 4) \geq 0\,,$$
which is clearly true if $|x|$ is small.

\section*{Acknowledgments}
E.C. acknowledges support from FAPERJ - Brazil. C.G.R. was supported by CAPES - Brazil.

\end{document}